\documentclass[12pt]{amsart}

\usepackage{float}
\usepackage{amsmath,amssymb,amsthm}
\usepackage{tikz}
\usepackage{hyperref}

\newtheorem{theorem}{Theorem}[section]
\newtheorem{lemma}[theorem]{Lemma}
\newtheorem{prop}[theorem]{Proposition}

\newtheorem{cor}[theorem]{Corollary}
\newtheorem{corollary}[theorem]{Corollary}

\theoremstyle{remark}
\newtheorem{definition}[theorem]{Definition}

\renewcommand\subset{\subseteq}

\newcommand{\cP}{\mathcal{P}}
\newcommand{\et}{\quad\mbox{and}\quad}

\newcommand{\bQ}{\mathbb{Q}}
\newcommand{\bR}{\mathbb{R}}

\newcommand{\bZ}{\mathbb{Z}}
\newcommand{\cO}{{\mathcal{O}}}
\newcommand{\cU}{{\mathcal{U}}}
\newcommand{\disp}{\displaystyle}
\newcommand{\End}{\mathrm{End}}
\newcommand\GrO{\mathcal{O}}
\newcommand{\hlambda}{\widehat{\lambda}}

\newcommand{\norm}[1]{\|#1\|}
\newcommand{\Span}[1]{\langle\,#1\rangle}

\renewcommand\theta{\vartheta}
\newcommand{\ua}{\mathbf{a}}
\newcommand{\ualpha}{\boldsymbol{\alpha}}
\newcommand{\ud}{\mathbf{d}}
\newcommand{\uu}{\mathbf{u}}
\newcommand{\uv}{\mathbf{v}}
\newcommand{\ux}{\mathbf{x}}
\newcommand{\uy}{\mathbf{y}}
\newcommand{\uz}{\mathbf{z}}

\newcommand{\UU}{\mathcal{U}}

\def\bx{{\bf x}}

\numberwithin{equation}{section}

\begin{document}

\baselineskip=17pt

\title[Simultaneous rational approximation]%
{Simultaneous rational approximation to\\ successive powers of a real number}

\author{Anthony Po\"els}
\address{
   D\'epartement de math\'ematiques\\
   Universit\'e d'Ottawa\\
   150 Louis Pasteur\\
   Ottawa, Ontario K1N 6N5, Canada}
\email{anthony.poels@uottawa.ca}
\author{Damien Roy}
\address{
   D\'epartement de math\'ematiques\\
   Universit\'e d'Ottawa\\
   150 Louis Pasteur\\
   Ottawa, Ontario K1N 6N5, Canada}
\email{droy@uottawa.ca}
\thanks{Work of both authors partially supported by NSERC}


\begin{abstract}
We develop new tools leading, for each integer $n\ge 4$, to a
significantly improved upper bound for the uniform exponent of rational
approximation $\hlambda_n(\xi)$ to successive powers $1,\xi,\dots,\xi^n$
of a given real transcendental number $\xi$. As an application, we obtain
a refined lower bound for the exponent of approximation to $\xi$ by
algebraic integers of degree at most $n+1$.  The new lower bound is
$n/2+a\sqrt{n}+4/3$ with $a=(1-\log(2))/2\simeq 0.153$, instead of the
current $n/2+\GrO(1)$.
\end{abstract}

\maketitle

\section{Introduction}
\label{sec:Intro}

In their seminal 1969 paper \cite{DS1969}, Davenport and Schmidt introduce
a novel approach to study how well a given real number $\xi$
may be approximated by algebraic integers of degree at most $n+1$, for a
given positive integer $n$.
Using geometry of numbers, they show that if, for some $c>0$ and $\lambda>0$,
there are arbitrarily large values of $X$ for which the conditions
\begin{equation}
\label{intro:eq1}
 |x_0|\le X \et \max_{1\le k\le n}|x_0\xi^k-x_k|\le cX^{-\lambda}
\end{equation}
admit no non-zero integer solution $\ux=(x_0,\dots,x_n)\in\bZ^{n+1}$,
then for some $c'>0$ there are infinitely many algebraic integers $\alpha$
of degree at most $n+1$ such that
\[
 |\xi-\alpha|\le c' H(\alpha)^{-1/\lambda-1}
\]
where $H(\alpha)$ stands for the height of $\alpha$, namely the largest
absolute value of the coefficients of its irreducible polynomial
over $\bZ$.  Assuming that $\xi$ is not itself an algebraic number of
degree at most $n$, they further show that admissible values for $\lambda$ are
$\lambda=1$ if $n=1$, $\lambda=(-1+\sqrt{5})/2$ if $n=2$, $\lambda=1/2$
if $n=3$ and $\lambda=1/\lfloor n/2\rfloor$ if $n\ge 4$.  Since then,
their results have been extended to many other settings which include
approximation to $\xi$ by algebraic integers of a given degree
\cite{BT2000}, by algebraic units of a given degree \cite{Te2001} and
by conjugate algebraic integers \cite{RW2004}.  A $p$-adic analog is given
in \cite{Te2002} and an extension to a variety of inhomogeneous problems
is proposed in \cite{BL2005a}.  Refined values for
$\lambda$ have also been established by Laurent \cite{La2003}, by
Schleischitz \cite{Schl2020, Schl2021} and by Badziahin \cite{Ba2021}.

For each $\xi\in\bR$ and each integer $n\ge 1$, let $\hlambda_n(\xi)$
(resp.\ $\lambda_n(\xi)$) denote the supremum of all $\lambda \ge 0$
such that, for $c=1$, the conditions \eqref{intro:eq1} admit a non-zero
integer solution $\bx=(x_0,\dots,x_n)\in\bZ^{n+1}$ for each sufficiently
large $X$ (resp.\ for arbitrarily large values of $X$).  Further, let
$\tau_{n+1}(\xi)$ denote the supremum of all $\tau\ge 0$ for which
there exist infinitely many algebraic integers $\alpha$ of degree at
most $n+1$ with $|\xi-\alpha| \le H(\alpha)^{-\tau}$. Then, in crude
form, the observation of Davenport and Schmidt is that
\begin{equation}
\label{intro:eq:dualite}
 \tau_{n+1}(\xi) \ge \hlambda_n(\xi)^{-1}+1.
\end{equation}
Thus any upper bound on $\hlambda_n(\xi)$ yields a lower bound on $\tau_{n+1}(\xi)$.

Assume now that $\xi$ is not itself an algebraic number of degree at most $n$,
namely that $1,\xi,\dots,\xi^n$ are linearly independent over $\bQ$, or equivalently
that $[\bQ(\xi):\bQ]>n$.  Then a result of Dirichlet \cite[\S II.1, Theorem 1A]{Schm1980}
yields $1/n\le \hlambda_n(\xi)$ and, expecting the equality, it is natural to conjecture that
$\tau_{n+1}(\xi) \ge n+1$ as in \cite[\S 5]{Schm1983}.  For $n=1$, we have
$\hlambda_1(\xi)=1$ and the conjectured lower bound $\tau_2(\xi)\ge 2$ follows.
However, for $n=2$, the upper bound $\hlambda_2(\xi)\le (-1+\sqrt{5})/2\cong 0.618$
from \cite[Theorem 1a]{DS1969} is best possible by \cite[Theorem 1.1]{Ro2004}, and the corresponding lower bound $\tau_3(\xi)\ge (3+\sqrt{5})/2\cong 2.618$ is also best possible by
\cite[Theorem 1.1]{Ro2003} (see also \cite{RZ2011}).  This disproves the natural
conjecture for $n=2$ and suggests that it might be false as well for each $n\ge 3$.
Any counterexample $\xi$ would have $[\bQ(\xi):\bQ]>n$ and $\hlambda_n(\xi)>1/n$.
So, it would be transcendental over $\bQ$  according to Schmidt's subspace theorem
\cite[\S VI.1, Corollary 1E]{Schm1980}.  Although the existence of such number remains
an open problem for $n\ge 3$, we know by contrast large families of transcendental real
numbers $\xi$ with $\hlambda_2(\xi)>1/2$.  In chronological order, they are the extremal
numbers $\xi$ of \cite{Ro2004}, the Sturmian continued fractions of \cite{BL2005b},
Fischler's numbers from \cite{Fi2007}, the Fibonacci type numbers
of \cite{Ro2007} and the Sturmian type numbers of \cite{Po2020}, all contained in the very
general class of numbers studied in \cite{Po2021}.  In particular, we know
by \cite[Corollary]{Ro2007} that the values $\hlambda_2(\xi)$ with $\xi$ real and
transcendental form a dense subset of the interval $[1/2,(-1+\sqrt{5})/2]$.

For $n\ge 3$, recent progresses have been made on upper bounds for $\hlambda_n(\xi)$.
The estimates $\hlambda_3(\xi)\le 1/2$ and $\hlambda_n(\xi)\le 1/\lfloor n/2\rfloor$ for
$n\ge 4$ from \cite[Theorems 2a and 4a]{DS1969} have been refined by Laurent \cite{La2003}
to $\hlambda_n(\xi)\le 1/\lceil n/2\rceil$ for each $n\ge 3$, together with an important
simplification in the proof.  When $n=3$, the best computed upper bound (yet not
optimal) remains that of \cite{Ro2008},
\begin{equation}
    \label{intro:eq:lambda_3}
    \hlambda_3(\xi)\leq \alpha =0.4245\cdots,
\end{equation}
where $\alpha$ is the root of the polynomial $1-3x+4x^3-x^4$ in the interval $[1/3,1/2]$.
For even integers $n=2m\ge 4$, Schleischitz \cite{Schl2020, Schl2021}
refined the upper bound $\hlambda_n(\xi)\le 1/m$ by reducing to the case
where $\lambda_n(\xi)\le 1/m$ and then by using a transference inequality
of Marnat and Moshchevitin \cite{MM2020} relating $\lambda_n(\xi)$ and
$\hlambda_n(\xi)$.
Nevertheless, all those refinements, including those of the recent preprint
of Baziahin \cite{Ba2021}, are of the form $\hlambda_n(\xi)\le 1/(n/2+c_n)$
with $0 < c_n < 1$.  Our main result below improves significantly
on this when $n$ is large.

\begin{theorem}
\label{intro:thm:main}
For any integer $n\ge 2$ and any $\xi\in\bR$ with $[\bQ(\xi):\bQ]>n$, we have
\begin{equation}
\label{intro:thm:main:eq}
    \hlambda_n(\xi) \leq \frac{1}{n/2+a\sqrt{n}+1/3}.
\end{equation}
where $a=(1-\log(2))/2\cong 0.1534$.
\end{theorem}

Note that the multiplicative constant $a$ in the denominator is not optimal and
could be improved with additional work.  The same applies to the additive
constant $1/3$ given the actual choice of $a$.  In view of \eqref{intro:eq:dualite},
this gives
\[
 \tau_{n+1}(\xi)\ge n/2+a\sqrt{n}+4/3
\]
for the same $n$ and $\xi$.

As explained by Bugeaud in \cite[Prop.~3.3]{Bu2004}, the arguments of Davenport and
Schmidt leading to \eqref{intro:eq:dualite} can also be adapted to Wirsing's problem
of approximating real numbers $\xi$ by algebraic numbers, yielding
$\omega^*_n(\xi)\ge \hlambda_n(\xi)^{-1}$ for any integer $n\ge 1$, where
$\omega^*_n(\xi)$ denotes the supremum of all $\omega>0$ for which there
exist infinitely many algebraic numbers $\alpha$ of degree at most $n$
with $|\xi-\alpha|\le H(\alpha)^{-\omega-1}$.  Thus
when $[\bQ(\xi):\bQ]>n$, the inequality \eqref{intro:thm:main:eq} implies that
$\omega_n^*(\xi)\ge n/2+a\sqrt{n}+1/3$.  However, this is superseded by the recent
breakthrough of Badziahin and Schleischitz who showed in \cite{BS2021} that
$\omega_n^*(\xi)>n/\sqrt{3}$ when $[\bQ(\xi):\bQ]>n\ge 4$.  Previous to their work,
the best lower bounds for large values of $n$ were of the form $n/2+\GrO(1)$.

For small values of $n$, namely for $n$ odd with $5\le n\le 49$ and for $n$ even
with $4\le n\le 100$, we obtain the following estimates which improve on Theorem \ref{intro:thm:main}.

\begin{theorem}
\label{intro:thm:impair}
Suppose that $n=2m+1\ge 5$ is odd. Then for each $\xi\in\bR$ with $[\bQ(\xi):\bQ]>n$, we have
\[
    \hlambda_{2m+1}(\xi) \leq \alpha_m,
\]
where $\alpha_m$ is the positive root of the polynomial $P_m(x)=1 -(m+1)x - mx^2$.
\end{theorem}

\begin{theorem}
\label{intro:thm:pair}
Suppose that $n =2m\ge 4$ is even. Then for each $\xi\in\bR$ with $[\bQ(\xi):\bQ]>n$, we have
\begin{equation*}
    \hlambda_{2m}(\xi) \leq \beta_m
\end{equation*}
where $\beta_m$ is the positive root of the polynomial
\[
Q_m(x)
 =\begin{cases}
    1-mx-mx^2-m(m-1)x^3 &\text{if $m\ge3$,}\\
    1-3x+x^2-2x^3-2x^4 &\text{if $m=2$.}
   \end{cases}
\]
\end{theorem}

Again, these upper bounds could be slightly improved with more work, at least for each $n\geq 6$
(even or odd).  Note that they are relatively close to $1/(m+2)$ as one finds
\[
 \frac{1}{m+2}<\alpha_m<\frac{1}{m+2}+\frac{2}{(m+2)^3}
 \et
 \frac{1}{m+2}<\beta_m<\frac{1}{m+2}+\frac{7}{(m+2)^3}
\]
for each $m\ge 2$.  The table below shows how they compare to those of Laurent (L.) for odd $n\le 13$
and to those of Schleischitz (S.) and Badziahin (B.) for even $n\le 12$.

\begin{figure}[H]
 \label{intro:table}
    {\renewcommand{\arraystretch}{1.2}
    \centering
        $\begin{array}{|c|| c | c | c |c|}
        \hline
        \mathbf{n}  & \mathbf{L.} & \mathbf{S.} & \mathbf{B.} &  \mathbf{new} \\ \hline
        4    &                 & 0.3706\cdots & 0.3660\cdots & 0.3370\cdots \\ \hline
        5    & 0.3333\cdots  & &                               & 0.2807\cdots \\ \hline
        6    &                 & 0.2681\cdots & 0.2637\cdots & 0.2444\cdots \\ \hline
        7    & 0.2500\cdots  & &                               & 0.2152\cdots \\ \hline
        8    &                 & 0.2107\cdots & 0.2071\cdots & 0.1919\cdots \\ \hline
        9    & 0.2000\cdots  & &                               & 0.1753\cdots \\ \hline
        10   &                & 0.1737\cdots & 0.1708\cdots  & 0.1587\cdots \\ \hline
        11   & 0.1666\cdots & &                                & 0.1483\cdots \\ \hline
        12   &                & 0.1478\cdots & 0.1454\cdots  & 0.1357\cdots \\ \hline
        13   & 0.1428\cdots & &                                & 0.1286\cdots \\ \hline
        \end{array}$}
   \caption{Upper bounds for $\hlambda_n$}
   \label{intro:table1}
\end{figure}

For the proof we develop new tools.  The main one concerns the behavior of the
function $f(\ell)=\dim\UU^\ell(A)$ for $\ell\in\{0,1,\dots,n+1\}$ where $A$ is any
subspace of $\bR^{n+1}$ and $\UU^\ell(A)$ stands for the subspace of
$\bR^{n-\ell+1}$ spanned by the images of $A$ through the projections
$(x_0,\dots,x_n)\mapsto (x_k,\dots,x_{k+n-\ell})$ for $k=0,\dots,\ell$,
with the convention that $\UU^{n+1}(A)=0$.
In Sections \ref{sec:three} and \ref{sec:proofs}, we show that such a function is
concave and monotone increasing as long as $\UU^\ell(A)\neq\bR^{n-\ell+1}$.
We also study the degenerate cases where $f(\ell)<\dim(A)+\ell \le n-\ell+1$.

In Section \ref{sec:minimal}, we form a sequence of minimal points
$(\ux_i)_{i\ge 0}$ for $\xi$ in $\bZ^{n+1}$, and recall how the exponents
$\hlambda_n(\xi)$ and $\lambda_n(\xi)$ can be computed from this data.
Given integers $0\le j,\ell\le n$,
we say that Property $\cP(j,\ell)$ holds if, for any subspace $A=\Span{\ux_i,\dots,\ux_q}$
of dimension at most $j+1$ spanned by consecutive minimal points with a large enough
initial index $i$, we have $\dim\UU^\ell(A)\ge\dim(A)+\ell$.  This is a crucial notion
with the remarkable feature that $\cP(j,\ell)$ implies $\cP(j+1,\ell-1)$ when $\ell\ge 1$.

In Sections \ref{sec:P0} and \ref{sec:P1}, we establish some consequences of Properties
$\cP(0,\ell)$ and $\cP(1,\ell)$ respectively and we provide lower bounds on
$\hlambda_n(\xi)$ which ensure that these properties hold.  In Section \ref{sec:P0}, we
also study the general situation where $\cP(j,\ell-1)$ holds but not $\cP(j,\ell)$ for
some integer $\ell\ge 1$.  In Sections \ref{sec:first} and
\ref{sec:alt}, we provide two types of upper bounds for the height of $\UU^\ell(A)$
when $\cP(j,\ell)$ holds and $A=\Span{\ux_i,\dots,\ux_q}$ has dimension $j+1$.
In particular, the estimate of Section \ref{sec:alt} yields a strong constraint on the growth
of the norms $X_i=\norm{\ux_i}$.  These tools are combined in Section \ref{sec:main} to
prove Theorem \ref{intro:thm:main}.
Finally Theorems \ref{intro:thm:impair} and \ref{intro:thm:pair} are proved respectively
in Sections~\ref{sec:odd} and~\ref{sec:even}, with the help of a new construction
presented in Section~\ref{sec:new}.

We start in the next section by fixing some notation, including our notion of height
for the subspaces  of $\bR^m$ defined over $\bQ$.

%
%

\section{Heights}
\label{sec:heights}

For each integer $m\ge 1$, we view $\bR^m$ as an Euclidean space
for the usual scalar product of points $\ux,\,\uy\in\bR^m$ written
$\ux\cdot\uy$, and we denote by $\norm{\ux}=\sqrt{\ux\cdot\ux}$
the Euclidean norm of a point $\ux\in\bR^m$.  For
each integer $k=1,\dots,m$, we also identify $\bigwedge^k\bR^m$
with $\bR^{\binom{m}{k}}$ via a choice of ordering of the Pl\"ucker
coordinates and we denote by $\norm{\ualpha}$ the resulting 
Euclidean norm of a point $\ualpha\in\bigwedge^k\bR^m$.

For any subset $A$ of $\bR^m$, we denote by $\langle A\rangle$
the linear subspace of $\bR^m$ spanned by $A$ over $\bR$.  When
$A$ is a finite set $\{\ux_1,\dots,\ux_k\}$, we simply write
$\langle \ux_1,\dots,\ux_k \rangle$.

Let $V$ be an arbitrary vector subspace of $\bR^m$ defined over $\bQ$.
If $V\neq 0$, we define its height $H(V)$ as the covolume of the lattice
$V\cap\bZ^m$ inside $V$.  Explicitly, if $\dim(V)=k$ and if
$\{\ux_1,\dots,\ux_k\}$ is a basis of $V\cap\bZ^m$ over $\bZ$, then
\begin{equation}
\label{heights:eq:H(V)}
 H(V)=\norm{\ux_1\wedge\cdots\wedge\ux_k}.
\end{equation}
For $V=0$, we set $H(0)=1$.  Then, we have the duality relation
\begin{equation}
\label{heights:eq:duality}
 H(V)=H(V^\perp)
\end{equation}
where $V^\perp$ denotes the orthogonal complement of $V$ in $\bR^m$
\cite[Chapter I, \S8]{Schm1991}.  In particular, this gives
$H(\bR^m)=H(0^\perp)=1$.  If $\ux\in\bZ^m$ is \emph{primitive}, namely
if the gcd of its coordinates is $1$, we have
$H(\langle\ux\rangle)=H(\langle\ux\rangle^\perp)=\norm{\ux}$.
We will also need the following important inequality of Schmidt
\begin{equation}
\label{heights:eq:schmidt}
 H(U\cap V)H(U+V)\le H(U)H(V),
\end{equation}
valid for any subspaces $U$ and $V$ of $\bR^m$
defined over $\bQ$ \cite[Chapter I, Lemma 8A]{Schm1991}.

Finally, given $\xi\in\bR$ and $\ux=(x_0,\dots,x_m)\in
\bZ^{m+1}\setminus\{0\}$, we define
\[
 L_\xi(\ux)=\max_{1\le j\le m}|x_0\xi^j-x_j|,
\]
and note that
\[
 L_\xi(\ux) \asymp \norm{\Xi_m\wedge\ux}
 \quad
 \text{where}
 \quad
 \Xi_m=(1,\xi,\dots,\xi^{m})
\]
with implicit constants depending only on $\xi$ and $m$.
The latter relation is instructive since the product
$\norm{\Xi_m\wedge\ux}\,\norm{\Xi_m}^{-1}\norm{\ux}^{-1}$ represents
the sine of the angle between $\Xi_m$ and $\ux$.  We will
repeatedly use the following generalization of \cite[Lemma 9]{DS1969}.

\begin{lemma}
\label{heights:lemma}
Suppose that $\ux_1,\dots,\ux_k\in\bZ^{m+1}$ are linearly
independent.  Then
\[
 H(\langle\ux_1,\dots,\ux_k\rangle)
 \le \norm{\ux_1\wedge\cdots\wedge\ux_k}
 \ll \sum_{i=1}^k\norm{\ux_i}\prod_{j\neq i}L_\xi(\ux_j)
\]
with an implied constant which depends only on $\xi$ and\/ $m$.
\end{lemma}

This follows from \eqref{heights:eq:H(V)} by writing $\ux_j=x_{j,0}\Xi_m+\Delta_j$
for $j=1,\dots,k$, where $x_{j,0}$ stands for the first coordinate of
$\ux_j$, and then by expanding the exterior product upon noting that
$\norm{\Delta_j}\asymp L_\xi(\ux_j)$.

%
%

\section{Three crucial propositions}
\label{sec:three}

Let $\ell,n$ be integers with $0\le \ell\le n$.  For each
$\ux=(x_0,\dots,x_n)\in\bR^{n+1}$, we denote by
\[
 \ux^{(k,\ell)} = (x_{k},\dots,x_{k+n-\ell})\in \bR^{n+1-\ell}
 \quad (0\le k\le \ell),
\]
the points consisting of $n+1-\ell$ consecutive coordinates of $\ux$,
and we denote by
\[
  \UU^{\ell}(\ux) = \Span{ \ux^{(0,\ell)},\dots, \ux^{(\ell,\ell)}}
  \subseteq \bR^{n+1-\ell},
\]
the vector subspace of $\bR^{n+1-\ell}$ which they generate.
In general, for each non-empty subset $A$ of $\bR^{n+1}$, we define
\[
 \UU^\ell(A)=\sum_{\ux\in A}\UU^\ell(\ux)
 \et
 \UU^{n+1}(A)=0.
\]
Then, we have $\UU^k(\UU^{\ell-k}(A)) = \UU^\ell(A)
= \UU^\ell(\langle A\rangle)$ for any integers
$0\le k\le \ell\le n+1$.

Our interest in the truncated points $\ux^{(k,\ell)}$ comes from
the fact that, when $\ux\in\bZ^{n+1}$ and $\ell<n$, they belong to
$\bZ^{n+1-\ell}$ and, for given $\xi\in \bR$, they satisfy
\begin{equation}
 \label{three:eq:xkl}
 \norm{\ux^{(k,\ell)}} \le \norm{\ux}
 \et
 L_\xi(\ux^{(k,\ell)}) \ll L_\xi(\ux)
\end{equation}
with implied constants depending only on $\xi$ and $n$.
So Lemma \ref{heights:lemma} yields
\begin{equation}
 \label{three:eq:HUx}
 H(\UU^\ell(\ux))\ll \norm{\ux}L_\xi(\ux)^{d-1}
 \quad
 \text{if}
 \quad
 d=\dim\UU^\ell(\ux)>0.
\end{equation}
In general, when $A$ is a subspace of $\bR^{n+1}$ defined
over $\bQ$, the subspace $\UU^\ell(A)$ of $\bR^{n+1-\ell}$
is also defined over $\bQ$ and, as the above example shows,
we need some information on its dimension in order to estimate
its height.

In this section, we state three propositions concerning
$\UU^\ell(A)$ as a function of $\ell$, for
a fixed subspace $A$ of $\bR^{n+1}$, but postpone their proofs
to the next section.   In order to state the first one,
we recall that a function $f:\{0,\dots,n+1\}\rightarrow \bR$ is
\emph{convex} if it satisfies the following equivalent conditions
\begin{itemize}
 \item[(1)] $f(i)-f(i-1) \le f(i+1)-f(i)$ \ for $i=1,\dots,n$;
 \medskip
 \item[(2)] $\disp \frac{f(j)-f(i)}{j-i} \le \frac{f(k)-f(j)}{k-j}$ \
 whenever $0\leq i<j<k\le n+1$.
\end{itemize}
We say that $f$ is \emph{concave} if $-f$ is convex. We also fix a
positive integer $n$.

\begin{prop}
\label{three:prop:concave}
Let $A$ be a subspace of $\bR^{n+1}$. Then $f(\ell)=\dim\UU^\ell(A)$
is a concave function of $\ell\in\{0,\dots,n+1\}$.  Moreover,
there is an integer $m\in\{0,\dots,n+1\}$ for which the function $f$ is
monotonically increasing on $\{0,\dots,m\}$, while strictly
decreasing with $f(\ell)=n-\ell+1$ for $\ell\in\{m,\dots,n+1\}$.
\end{prop}

Figure \ref{three:fig} illustrates this result.  Taking it for granted,
we deduce a useful corollary.

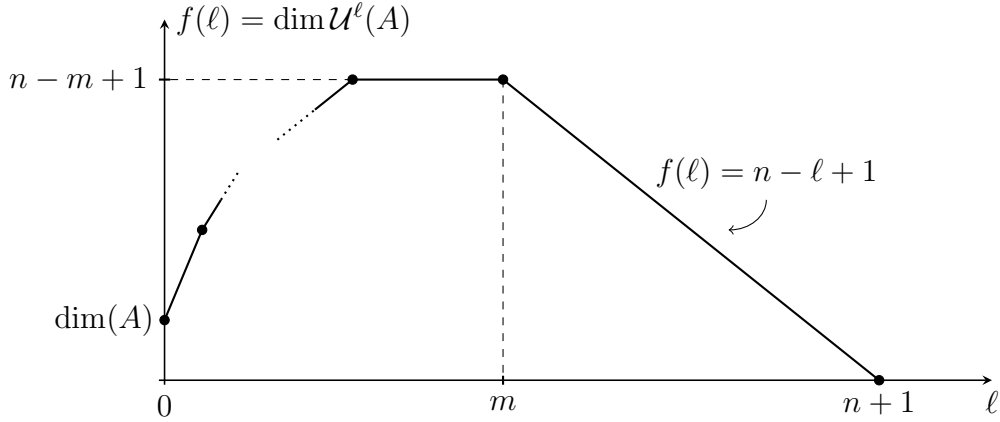
\begin{figure}[H]
 \begin{tikzpicture}[xscale=0.5,yscale=0.4]
       \draw[-stealth, semithick] (-0.15,0)--(22,0) node[below]{$\ell$};
       \draw[-stealth, semithick] (0,-0.15)--(0,12) node[right]{$f(\ell)=\dim\UU^\ell(A)$};
       \node[below] at (0,-0.15) {$0$};
       \node[draw,circle,inner sep=1.25pt,fill] at (0,2) {};
       \node[draw,circle,inner sep=1.25pt,fill] at (1,5) {};
       \node[draw,circle,inner sep=1.25pt,fill] at (5,10) {};
       \node[draw,circle,inner sep=1.25pt,fill] at (9,10) {};
       \node[draw,circle,inner sep=1.25pt,fill] at (19,0) {};
       \node[left] at (0,2) {$\dim(A)$};
       \node[below] at (19,0) {$n+1$};
       \draw[thick] (0,2)--(1,5)--(1.5,6);
       \draw[thick,dotted] (1.5,6)--(2,7);
       \draw[thick,dotted] (3,8)--(4,9);
       \draw[thick] (4,9)--(5,10)--(9,10)--(19,0);
       \draw[dashed] (9,10) -- (9,0);
       \draw[thick] (9,0.15)--(9,-0.15) node[below]{$m$};
       \draw[dashed] (5,10) -- (0,10);
       \draw[thick] (0.15,10)--(-0.15,10) node[left]{$n-m+1$};
       \draw[->] (16,6) to [out=270,in=0, looseness=1] (15,5);
       \node[above] at (16,6) {$f(\ell)=n-\ell+1$};
 \end{tikzpicture}
\caption{Graph of the piecewise linear function interpolating the
values $f(\ell)=\dim\UU^\ell(A)$ at integers $\ell\in\{0,\dots,n+1\}$.}
\label{three:fig}
\end{figure}

\begin{corollary}
\label{three:cor}
Let $A$ be a subspace of $\bR^{n+1}$, and let $\ell\in\{1,\dots,n\}$. Then
\begin{itemize}
\item[(i)]
    $\min\{\dim\UU^{\ell}(A),\,\dim(A)+\ell-1\} \le \dim\UU^{\ell-1}(A)$;
\item[(ii)]
    $\min\{\dim\UU^{\ell-1}(A),\,n-\ell+1\} \le \dim\UU^{\ell}(A)$.
\end{itemize}
\end{corollary}

\begin{proof}
Let $f$ and $m$ be as in Proposition \ref{three:prop:concave}.
If $f(\ell-1)=n-\ell+2$, then $\ell-1\ge m$, so $f(\ell)=n-\ell+1$
and we are done.  Otherwise the function $f$ is concave and
monotonically increasing on $\{0,\dots,\ell\}$, hence $f(\ell-1)
\le f(\ell)$, so (ii) holds. If $f(\ell-1)=f(\ell)$ or if
$\ell=1$, then (i) also holds since $f(0)=\dim\UU^0(A)=\dim(A)$.
So we may further assume that $f(\ell-1) <  f(\ell)$ and that
$\ell > 1$. By concavity of $f$, we deduce that
\begin{align*}
    \frac{f(\ell-1)-f(0)}{\ell-1} \geq f(\ell)-f(\ell-1) \geq 1,
\end{align*}
hence $f(\ell-1)\geq f(0)+\ell-1 = \dim(A)+\ell-1$,
and (i) holds again.
\end{proof}

The second proposition provides additional information in the
degenerate situation where $\dim\UU^\ell(A)<\dim(A)+\ell$.

\begin{prop}
 \label{three:prop:height}
Let $j,\ell\ge 0$ be integers with $j+2\ell\le n$, and let
$A$ be a subspace of $\bR^{n+1}$ of dimension $j+1$ defined
over $\bQ$. Suppose that
\[
 d:=\dim\UU^\ell(A) \le j+\ell,
\]
and set $V=\UU^{n-d}(A)$.  Then, we have $0\le d-j-1< \ell\le n-d$ and
\[
 \dim\UU^t(A)=d
 \et
 H(\UU^t(A))\asymp H(V)^{n-d-t+1}
\]
for each $t=d-j-1,\dots,n-d$, with implied constants depending
only on $n$.  Moreover, for such
$t$ and for $\ux\in\bR^{n+1}$, the condition
$\UU^t(\ux)\subseteq\UU^t(A)$ is equivalent to
$\UU^{n-d}(\ux)\subseteq V$, thus independent of $t$.
\end{prop}

The last result exhibits the generic behavior of a family of linear maps.

\begin{prop}
 \label{three:prop:avoiding}
Let $\ell\in\{0,\dots,n\}$ and let $V$ be a subspace
of $\bR^{n-\ell+1}$.  Suppose that a vector subspace $A$ of
$\bR^{n+1}$ satisfies $\dim(A)\le n-\ell+1$
and $\UU^\ell(A)\not\subseteq V$.  Then, there exists
a point $\ua=(a_0,\dots,a_\ell)\in\bZ^{\ell+1}$ with
$\sum_{k=0}^\ell |a_k|\le (n+1)^\ell$
such that the linear map
\begin{equation}
 \label{three:eq:tau}
 \begin{array}{rcl}
  \tau_\ua\colon \bR^{n+1} &\longrightarrow &\bR^{n-\ell+1}\\
  \ux&\longmapsto&\sum_{k=0}^\ell a_k\ux^{(k,\ell)}
 \end{array}
\end{equation}
is injective on $A$ with $\tau_\ua(A)\not\subseteq V$.
\end{prop}

%
%

\section{Proofs of the three propositions}
\label{sec:proofs}

Our goal is to prove the statements of the preceding section by
re-interpreting them in a polynomial setting similar to
that of \cite{RW2004}.  In particular, we will connect the
function $f$ of Proposition \ref{three:prop:concave} to
the Hilbert-Samuel function of a graded module over a
polynomial ring in two variables.  To this end, we start
by fixing some notation.

Let $E=\bR[T_0,T_1]$ denote the ring of polynomials in two variables
$T_0$ and $T_1$ over $\bR$, and let $D=\bR[\delta_0,\delta_1]$
denote the subring of $\End_\bR(E)$ spanned by the partial derivatives
\[
 \delta_0=\frac{\partial}{\partial T_0}
 \et
 \delta_1=\frac{\partial}{\partial T_1}
\]
restricted to $E$.  It is easily seen that these commuting linear
operators are algebraically independent over $\bR$.  Thus, $D$ is
a commutative ring isomorphic to $E$.  In particular, both $D$
and $E$ are graded rings (by the degree) as well as unique
factorization domains.  Moreover, $E$ is a $D$-module for the
natural action of the differential operators of $D$ on $E$.

For each integer $n\ge 0$, we denote by $E_n=\bR[T_0,T_1]_n$ and
$D_n=\bR[\delta_0,\delta_1]_n$ the homogeneous parts of $E$ and
$D$ of degree $n$.  We also denote by $\psi_n\colon\bR^{n+1}\to
E_n$ the linear isomorphism sending a point $\ux=(x_0,\dots,x_n)
\in\bR^{n+1}$ to the polynomial
\[
 \psi_n(\ux) = \sum_{i=0}^n \binom{n}{i} x_i T_0^{n-i}T_1^i.
\]
When $n\ge 1$, we find that
\begin{equation}
 \label{proofs:eq0}
 \delta_0\psi_n(\ux)=n\psi_{n-1}(\ux^{(0,1)})
 \et
 \delta_1\psi_n(\ux)=n\psi_{n-1}(\ux^{(1,1)}),
\end{equation}
thus $D_1\psi_n(\ux)=\psi_{n-1}(\UU^1(\ux))$.  We deduce that,
for any subspace $A$ of $\bR^{n+1}$, we have $D_1\psi_n(A)
= \psi_{n-1}(\UU^1(A))$ and so, by induction,
\begin{equation}
 \label{proofs:eq1}
 D_\ell\psi_n(A) = \psi_{n-\ell}(\UU^\ell(A))
 \quad \text{for each $\ell\in\{0,\dots,n\}$.}
\end{equation}
Thus, if we identify $\bR^{n+1}$ with $E_n$ for each $n\ge 0$,
then $\UU^\ell(A)$ becomes simply $D_\ell A$ for each subspace
$A$ of $E_n$ and each $\ell=0,\dots,n+1$, including $\ell=n+1$
because $D_{n+1}A=0$.

From now on, we fix a positive integer $n$, a subspace $A$
of $\bR^{n+1}$, and a spanning set $\{\ux_1,\dots,\ux_s\}$
of $A$ as a vector space over $\bR$. We set $P_i=\psi_n(\ux_i)$
for each $i=1,\dots,s$, and form the $D$-module homomorphism
$\varphi\colon D^s \to E$ given by
\[
 \varphi(\ud) = d_1P_1+\cdots+d_sP_s
\]
for each $\ud=(d_1,\dots,d_s)\in D^s$.  Then $M:=\ker(\varphi)$
is a graded submodule of $D^s$.  Define
\[
 f(\ell)=\dim(\varphi(D_\ell^s))
 \et
 g(\ell)=\dim(M_\ell)
 \quad
 \text{for each $\ell\in\{0,1,\dots,n+1\}$,}
\]
where $M_\ell=M\cap D_\ell^s$ stands for the homogeneous part of $M$
of degree $\ell$, and the dimensions are taken over $\bR$.  Then, we have
\begin{equation}
 \label{proofs:eq2}
 f(\ell)=\dim(D_\ell^s)-\dim(M_\ell)=(\ell+1)s-g(\ell)
 \quad
 (0\le \ell\le n+1).
\end{equation}
By \eqref{proofs:eq1}, we also have $\varphi(D_\ell^s)=D_\ell\psi_n(A)
=\psi_{n-\ell}(\UU^\ell(A))$ for $\ell=0,\dots,n$. Comparing
dimensions, this gives
\begin{equation}
 \label{proofs:eq3}
 f(\ell)=\dim(\UU^\ell(A)) \quad
 (0\le \ell\le n+1)
\end{equation}
upon noting that for $\ell=n+1$ both sides vanish.  Finally, we define
\begin{equation}
 \label{proofs:eq4}
 h(\ell)=g(\ell+1)-g(\ell)=\dim(M_{\ell+1}/\delta_1M_\ell)
 \quad
 \text{for each $\ell\in\{0,1,\dots,n\}$.}
\end{equation}
With this notation, our main observation is the following.

\begin{lemma}
\label{proofs:lemma:h}
For each $\ell\in\{1,\dots,n\}$, we have $h(\ell-1)\le h(\ell)$
with equality if and only if $M_{\ell+1}=D_1M_\ell$.
\end{lemma}

\begin{proof}
Fix an integer $\ell\in\{1,\dots,n\}$, and consider the linear
map $\nu\colon M_\ell \to M_{\ell+1}/\delta_1M_\ell$ given by
$\nu(\ud)=\delta_0\ud+\delta_1M_\ell$ for each $\ud\in M_\ell$.
If $\ud\in\ker(\nu)$, then $\delta_0\ud=\delta_1\uu$ for some
$\uu\in M_\ell\subseteq D_\ell^s$.  Hence $\delta_0$ divides
$\delta_1\uu$ in $D^s$, so $\uu=\delta_0\uv$ for some
$\uv\in D_{\ell-1}^s$, and then $\ud=\delta_1\uv$.  Since
$\ud,\uu\in M_\ell$, we find $0=\varphi(\uu)=\delta_0\varphi(\uv)$
and $0=\varphi(\ud)=\delta_1\varphi(\uv)$, thus $\varphi(\uv)=0$.
This means that $\uv\in M_{\ell-1}$, and so $\ud=\delta_1\uv
\in\delta_1M_{\ell-1}$.  This shows
that $\ker(\nu)\subseteq\delta_1M_{\ell-1}$.  As the reverse
inclusion is clear, we conclude that $\ker(\nu)=\delta_1M_{\ell-1}$,
and so $\nu$ induces an injective map from $M_\ell/\delta_1M_{\ell-1}$
to $M_{\ell+1}/\delta_1M_\ell$.  Comparing dimensions, we deduce
that $h(\ell-1)\le h(\ell)$.  Moreover, we have the equality
$h(\ell-1)=h(\ell)$ if and only if $\nu$ is surjective, a condition
which amounts to $M_{\ell+1}=\delta_0M_\ell+\delta_1M_\ell$ or
equivalently to $M_{\ell+1}=D_1M_\ell$.
\end{proof}

As a consequence, we deduce the first assertion of
Proposition \ref{three:prop:concave}.

\begin{corollary}
 \label{proofs:cor:lemma:h}
The function $g(\ell)=\dim(M_\ell)$ is convex on
$\{0,1,\dots,n+1\}$, while the function
$f(\ell)=\dim(\UU^\ell(A))$ is concave on the same set.
\end{corollary}

\begin{proof}
The assertion for $g$ follows directly from the lemma
and the definition of $h$ in \eqref{proofs:eq4}.  Then
\eqref{proofs:eq2} gives $f(\ell)$ as the sum of two
concave functions of $\ell$ on $\{0,\dots,n+1\}$,
namely $(\ell+1)s$ and $-g(\ell)$, thus $f$ is concave.
\end{proof}

For each $m\ge 0$, we note that the action of $D$ on $E$
induces a non-degenerate bilinear form
\[
 \begin{array}{rcl}
   D_m\times E_m &\longrightarrow &\bR\\
   (\delta,P) &\longmapsto &\delta P
 \end{array}
\]
which identifies $D_m$ with the dual of $E_m$, the
dual of the natural basis $(T_0^{m-i}T_1^i)_{0\le i\le m}$
of $E_m$ being
\[
 \left(
   \frac{\delta_0^{m-i}\delta_1^i}{(m-i)!i!}
 \right)_{0\le i\le m}.
\]
So, for each subspace $W$ of $E_m$ (resp.\ $W$ of $D_m$), its orthogonal
space
\[
 W^\perp=\{\delta\in D_m\,;\,\delta W=0\}\subseteq D_m
 \quad
 \big(\text{resp.\ } W^\perp=\{P\in E_m\,;\,WP=0\}\subseteq E_m\,\big)
\]
satisfies $\dim(W^\perp)=m+1-\dim(W)$ and $(W^\perp)^\perp=W$.
We can now prove the following.

\begin{lemma}
\label{proofs:lemma:injective}
Let $W$ be a proper subspace of $E_m$ for some $m\ge 0$.  Then
there are at most $m$ real numbers $a$ for which the differential
operator $\delta=\delta_0+a\delta_1$ is not injective on $W$.
\end{lemma}

\begin{proof}
Without loss of generality, we may assume that $W$ has codimension
$1$ in $E_m$.  Then $W^\perp=\langle \gamma\rangle$ for some non-zero
$\gamma\in D_m$.  Suppose that $\delta=\delta_0+a\delta_1\in D_1$ is not
injective on $W$, for some $a\in\bR$.  Then $\delta W$ is a proper
subspace of $E_{m-1}$ and so $\beta\delta W=0$ for some non-zero
$\beta\in D_{m-1}$.  Therefore $\beta\delta$ belongs to $W^\perp$, and
so is proportional to $\gamma$.  This means that $\delta$ divides
$\gamma$ in $D$.  As $\gamma$ has degree $m$, it admits at most
$m$ non-associate divisors of degree $1$.  Thus $a$ belongs to
a set of at most $m$ numbers.
\end{proof}

The next result is the second part of Proposition \ref{three:prop:concave}.

\begin{corollary}
\label{proofs:cor:lemma:injective}
There is a smallest integer $m\in\{0,\dots,n+1\}$ for which
$f(m)=n-m+1$.  For this choice of $m$, the function $f$ is
monotonically increasing on $\{0,\dots,m\}$, while strictly
decreasing with $f(\ell)=n-\ell+1$ for $\ell\in\{m,\dots,n+1\}$.
\end{corollary}

\begin{proof}
The existence of $m$ follows from the fact that $f(n+1)=0$.
For that $m$, we have $\varphi(D_m^s)=E_{n-m}$.  Thus for each
$\ell\in\{m,\dots,n+1\}$, we find $\varphi(D_\ell^s) =
D_{\ell-m}E_{n-m} = E_{n-\ell}$ and so $f(\ell)=n-\ell+1$.
It remains to prove that $f$ is monotonically increasing on
$\{0,\dots,m\}$.  This is automatic if $m=0$.  Otherwise,
since $f$ is concave on $\{0,\dots,n+1\}$, this amounts to showing
that $f(m-1)\le f(m)$.  By the choice of $m$, the vector space
$W=\varphi(D_{m-1}^s)$ is a proper subspace of $E_{n-m+1}$.
Then, by Lemma \ref{proofs:lemma:injective}, there is some
$\delta\in D_1$ which is injective on $W$, thus $f(m)=\dim(D_1W)\ge
\dim(\delta W)=\dim(W)=f(m-1)$.
\end{proof}

Similarly, we will derive Proposition \ref{three:prop:avoiding}
from the following result.

\begin{prop}
 \label{proofs:prop:avoiding}
Let $\ell\in\{0,\dots,n\}$ and let $S$ be a subspace
of $E_{n-\ell}$.  Suppose that a subspace $B$ of
$E_n$ satisfies $\dim(B)\le n-\ell+1$
and $D_\ell B\not\subseteq S$.  Then, there exist
$a_0,\dots,a_\ell\in\bZ$ with
$\sum_{k=0}^\ell |a_k| \le (n+1)^\ell$
such that the differential operator
$\delta=\sum_{k=0}^\ell a_k\delta_0^{\ell-k}\delta_1^k\in D_\ell$
is injective on $B$ with $\delta B\not\subseteq S$.
\end{prop}

\begin{proof}
We proceed by induction on $\ell$.  For $\ell=0$, the result
is automatic, it suffices to take $a_0=1$.  Suppose that
$\ell\ge 1$ and set $S^*=\{P\in E_{n-\ell+1}\,;\,D_1P\subseteq S\}$.
Then $S^*$ is a subspace of $E_{n-\ell+1}$ and $D_{\ell-1}B
\not\subseteq S^*$.  So, by induction, we may assume the existence
of $a^*_0,\dots,a^*_{\ell-1}\in\bZ$ with
$\sum_{k=0}^{\ell-1} |a^*_k| \le (n+1)^{\ell-1}$ for which $\delta^* =
\sum_{k=0}^{\ell-1} a^*_k\delta_0^{\ell-1-k}\delta_1^k\in D_{\ell-1}$
is injective on $B$ with $\delta^* B\not\subseteq S^*$.  Then,
$W=\delta^*B$ is a proper subspace of $E_{n-\ell+1}$ because its
dimension is $\dim(B)\le n-\ell+1$.  Since $W\not\subseteq S^*$, we
also have $D_1W\not\subseteq S$, and so there is at most one
$a\in\bR$ for which $\gamma=\delta_0+a\delta_1$ satisfies
$\gamma W\subseteq S$.  By Lemma \ref{proofs:lemma:injective}, we
can therefore choose $a\in\bZ$ with $|a|\le n-\ell+1\le n$ such
that $\gamma$ is injective on $W$ with $\gamma W\not\subseteq S$.
Then $\delta=\gamma\delta^*$ has the required properties.
\end{proof}

To deduce Proposition \ref{three:prop:avoiding}, we simply apply
the above result with $S=\psi_{n-\ell}(V)$, $B=\psi_n(A)$ and
note that, by virtue of \eqref{proofs:eq0}, we have $\delta\circ\psi_n
= n(n-1)\cdots(n-\ell+1)\psi_{n-\ell}\circ\tau_\ua$, with
$\ua=(a_0,\dots,a_\ell)$.

Finally, for the proof of Proposition \ref{three:prop:height},
we need to extend the notion of height on the homogeneous
components of $D$ and $E$.  For each $m\ge 0$, we denote by
$\psi_m^*\colon\bR^{m+1}\to D_m$ the isomorphism which is
dual to $\psi_m\colon\bR^{m+1}\to E_m$ in the sense that
$\psi_m^*(\uy)\psi_m(\ux)=\uy\cdot\ux$ for any $\ux,\uy
\in\bR^{m+1}$. We say that a subspace $S$ of $D_m$
(resp.\ $E_m$) is defined over $\bQ$ if it is generated
by elements of $\bQ[\delta_0,\delta_1]_m$ (resp.\
$\bQ[T_0,T_1]_m$) or equivalently if $(\psi^*_m)^{-1}(S)$
(resp.\ $(\psi_m)^{-1}(S)$ is a subspace $V$ of $\bR^{m+1}$
defined over $\bQ$, and we define its height $H(S)$
to be $H(V)$.  Then the formula \eqref{heights:eq:duality}
translates into $H(S)=H(S^\perp)$ for each subspace $S$
of $E_m$ defined over $\bQ$, and its orthogonal space
$S^\perp$ in $D_m$.  Moreover, for each non-zero
$\delta\in \bQ[\delta_0,\delta_1]_m$ and each $k\ge 0$,
the subspace $D_k\delta$ of $D_{k+m}$ is defined over $\bQ$
of dimension $k+1$ and, by \cite[Proposition 5.2]{RW2004},
its height satisfies
\begin{equation}
\label{proofs:eq5}
 H(D_k\delta)\asymp H(\langle\delta\rangle)^{k+1}
\end{equation}
with implied constants that do not depend on $\delta$.

In this setting, Proposition \ref{three:prop:height} follows
immediately from the next result upon setting $B=\psi_n(A)$ and
noting that our choice of height yields $H(D_t B)=H(\UU^t(A))$
for each $t=0,\dots,n$.

\begin{prop}
 \label{proofs:prop:height}
Let $j,\ell\ge 0$ be integers with $j+2\ell\le n$, and let
$B$ be a subspace of $E_n$ of dimension $j+1$ defined over
$\bQ$. Set $d=\dim(D_\ell B)$ and suppose that $d \le j+\ell$.
Then, we have $0\le d-j-1< \ell\le n-d$ and there exists
a non-zero operator $\delta\in \bQ[\delta_0,\delta_1]_d$
such that
\[
 D_tB=(D_{n-d-t}\delta)^\perp,
 \quad
 \dim(D_tB)=d
 \et
 H(D_tB)\asymp H(\langle\delta\rangle)^{n-d-t+1}
\]
for $t=d-j-1,\dots,n-d$, with implied constants depending
only on $n$.  Moreover, for such $t$ and for $P\in E_n$, the
condition $D_tP\subseteq D_tB$ is independent of\/ $t$ and
amounts to $\delta P=0$.
\end{prop}

\begin{proof}
We may assume that $B=\psi_n(A)$ (with $A$ defined over
$\bQ$), and then $f(t)=\dim(D_tB)$ for $t=0,\dots,n+1$.
Let $m$ be as in Corollary
\ref{proofs:cor:lemma:injective}, so that $f$ is
monotonically increasing on $\{0,\dots,m\}$.  Since
$f(\ell) = d \le j+\ell < n-\ell+1$, this corollary gives
$\ell<m$.  We deduce that
\[
 f(0)=j+1 \le f(\ell)=d \le f(m)=n-m+1,
\]
thus $1\le d-j$, while the hypotheses yield $d-j\le\ell\le n-d$.
In turn this gives
\[
 f(d-j) \le f(\ell)=f(0)+(d-j)-1.
\]
So, $f$ is not strictly increasing on $\{0,\dots,d-j\}$.  Being
concave, it is therefore constant on $\{d-j-1,\dots,m\}$, equal
to $f(\ell)=d$. In particular, we obtain $m=n+1-f(m)=n-d+1$,
and so $f(n-d)=d$.  This means that $D_{n-d}B$ has codimension
$1$ in $E_d$.  As it is defined over $\bQ$, we deduce that
$(D_{n-d}B)^\perp=\langle\delta\rangle$
for some non-zero $\delta\in \bQ[\delta_0,\delta_1]_d$.

For each $t=d-j-1,\dots,n-d$,
the subspace $D_tB$ of $E_{n-t}$ has dimension $d$, while
$D_{n-d-t}\delta$ has codimension $d$ in $D_{n-t}$.  Since
their product is
\[
 (D_{n-d-t}\delta)(D_tB)=\delta D_{n-d}B=0,
\]
we deduce that $D_tB=(D_{n-d-t}\delta)^\perp$.  Thus, using
\eqref{proofs:eq5}, we obtain
\[
 H(D_tB)=H(D_{n-d-t}\delta)
   \asymp H(\langle\delta\rangle)^{n-d-t+1}.
\]
Finally, for $P\in E_n$,
the condition $D_tP\subseteq D_tB$ may be rewritten as
$(D_{n-d-t}\delta)(D_t P)=0$, so it is equivalent to
$\delta P \in D_{n-d}^\perp=0$ (inside $E_{n-d}$).
\end{proof}

%
%

\section{Minimal points and properties $\cP(j,\ell)$}
\label{sec:minimal}

From now on, we fix a positive integer $n$ and a real number $\xi$ with
$[\bQ(\xi):\bQ] > n$.   Our goal is to establish an upper bound for
$\hlambda_n(\xi)$ which depends only on $n$.

Since $[\bQ(\xi):\bQ] > n$, non-zero points $\ux,\uy\in\bZ^{n+1}$
which satisfy $L_\xi(\ux)=L_\xi(\uy)<1$ have a non-zero first
coordinate and come by pairs $\uy=\pm \ux$.  Thus, for each large enough
real number $X\ge 1$, there is a unique pair of non-zero points $\pm \ux$
in $\bZ^{n+1}$ with $\norm{\ux}\le X$ for which $L_\xi(\ux)<1$ is minimal.
We choose the one whose first coordinate is positive and, like Davenport
and Schmidt in \cite{DS1969}, we call it the \emph{minimal point
corresponding to $X$}.  This differs slightly from their own definition,
but it plays the same role.

We order these minimal points in a
sequence $(\ux_i)_{i\ge 0}$ by increasing norm.  Then,
\begin{itemize}
 \item their norms $X_i=\norm{\ux_i}$ are positive and strictly increasing,
 \item the quantities $L_i=L_\xi(\ux_i)$ are strictly decreasing,
 \item if $L_\xi(\ux)<L_i$ for some $i\ge 0$ and some non-zero $\ux\in\bZ^{n+1}$,
       then $\norm{\ux}\ge X_{i+1}$.
\end{itemize}
In terms of the associated sequences $(X_i)_{i\ge 0}$ and $(L_i)_{i\ge 0}$,
we have the well-known formulas
\begin{equation}
 \label{minimal:eq:lambda}
  \lambda_n(\xi)
   = \limsup_{i\to\infty} \frac{-\log(L_i)}{\log(X_i)}
  \et
  \hlambda_n(\xi)
   = \liminf_{i\to\infty} \frac{-\log(L_i)}{\log(X_{i+1})}
\end{equation}
which follow from the definition of these exponents given in the
introduction.  In particular, if $\hlambda_n(\xi)>\lambda$ for some
$\lambda\in\bR$, then $L_i=o(X_{i+1}^{-\lambda})$ and a fortiori
\begin{equation}
 \label{minimal:eq:Li}
  L_i\ll X_{i+1}^{-\lambda},
\end{equation}
where from now on all implicit multiplicative constants are independent of $i$.

By construction, each minimal point $\ux_i$ is primitive and so we have
\[
 H(\Span{\ux_i})=X_i.
\]
For subspaces spanned by two consecutive minimal points, a simple adaptation
of the proofs of \cite[Lemma~2]{DS1967} and \cite[Lemma~4.1]{Ro2004}
yields the following estimate.

\begin{lemma}
 \label{minimal:lemma:HA1}
For each $i\ge 0$, we have $H\big(\Span{\ux_i,\ux_{i+1}}\big) =
\norm{\ux_i\wedge\ux_{i+1}}\asymp X_{i+1}L_i$.
\end{lemma}

More generally, we are
interested in the subspaces $\langle\ux_i,\dots,\ux_q\rangle$
of $\bR^{n+1}$ spanned by minimal points with consecutive indices,
as in \cite[\S 3]{NPR2020} (see also \cite{MM2020}).  It is well-known that,
for each $i\ge 0$, we have
\begin{equation}
\label{minimal:sum}
 \langle\ux_i,\ux_{i+1},\dots\rangle
 = \sum_{k=i}^\infty \langle \ux_k\rangle
 = \bR^{n+1}
\end{equation}
because these are subspaces of $\bR^{n+1}$ defined over $\bQ$ which
contain $\lim_{k\to\infty}\norm{\ux_k}^{-1}\ux_k=\norm{\Xi}^{-1}\Xi$
where $\Xi=(1,\xi,\dots,\xi^n)$ has $\bQ$-linearly independent coordinates.
This justifies the following construction.

\begin{definition}
\label{minimal:def:AY}
For each $i\ge 0$ and each $j=0,\dots,n-1$, we set
\[
 \sigma_j(i)=q,
 \quad
 A_j(i)=\langle\ux_i,\dots,\ux_q\rangle
 \et
 Y_j(i)=X_{q+1}
\]
where $q\ge i$ is the largest index for which $\dim\langle\ux_i,\dots,\ux_q\rangle=j+1$.
We also set
\[
 A_n(i)=\bR^{n+1} \et Y_{-1}(i)=X_i.
\]
\end{definition}

So, for $j=0,\dots,n-1$, we have
\[
 \dim(A_j(i))=j+1
 \et
 A_{j+1}(i)=\langle\ux_i,\dots,\ux_q,\ux_{q+1}\rangle
 \quad
 \text{where $q=\sigma_j(i)$.}
\]
For $j=0$, we note that $\sigma_0(i)=i$,  $A_0(i)=\langle\ux_i\rangle$
and $Y_0(i)=X_{i+1}$.  For $j\ge 1$, the following notation is useful.

\begin{definition}
\label{minimal:def:I}
We denote by $I$ the set of indices $i\geq 1$ such that
$\ux_{i-1},\ux_i,\ux_{i+1}$ are linearly independent.  We say
that $i<j$ are \emph{consecutive elements} of $I$ or that $j$
is the \emph{successor} of $i$ in $I$ if $j$ is the smallest
element of $I$ with $j>i$.
\end{definition}

When $n=1$, the set $I$ is empty.  However, when $n>1$, we
may form $q=\sigma_j(i)$ for each $i\ge 1$ and each
$j=1,\dots,n-1$. Since $\Span{\ux_{q-1},\ux_q}\subseteq
A_j(i)$ and $\ux_{q+1} \notin A_j(i)$, we deduce that
$q\in I$.  Thus $I$ is infinite.  Moreover, if $i<j$ are
consecutive elements of $I$, we have
\[
 \Span{\ux_i,\ux_{i+1}}=\dots=\Span{\ux_{j-1},\ux_j}
 \neq \Span{\ux_j,\ux_{j+1}},
\]
Applying Lemma \ref{minimal:lemma:HA1}, we obtain the following useful
estimate.

\begin{lemma}
 \label{minimal:XLXL}
Suppose that $n\ge 2$.  Then $I$ is an infinite set and for each pair $i<j$ of
consecutive elements of $I$, we have $X_jL_{j-1}\asymp X_{i+1}L_i$.
\end{lemma}

The above also shows that, for each integer $i\ge 0$, we have
$\sigma_1(i)=j$ where $j$ is the smallest element of $I$ with $j>i$, so
$A_1(i)=A_1(j-1)$ and $A_2(i)=A_2(j-1)=\Span{\ux_{j-1},\ux_j,\ux_{j+1}}$.

In the next sections, we will provide upper bound estimates for the height of
the subspaces $\UU^\ell(A_j(i))$ when the following condition is fulfilled.

\begin{definition}
\label{minimal:def:Pjl}
Let $j,\ell\in\{0,\dots,n\}$.  We say that property $\cP(j,\ell)$ holds
if, for each sufficiently large integer $i\ge 0$ and each $m=0,\dots,j$, we have
$\dim\UU^\ell(A_m(i))\ge m+\ell+1$.
\end{definition}

Of course, this depends on our fixed choice of $\xi$ and $n$.  Clearly
$\cP(n,0)$ holds, because $\UU^0(A_m(i))=A_m(i)$
has dimension $m+1$ for each $i\ge 0$ and each $m=0,\dots,n$.
Moreover $\cP(j,\ell)$ implies $\cP(j-1,\ell)$ if $j>0$.  The next result
provides a further crucial implication.

\begin{prop}
 \label{minimal:prop:Pjl}
Suppose that property $\cP(j,\ell)$ holds for some $j,\ell\in\{0,\dots,n\}$.
Then we have $j+2\ell\le n$.  If moreover $\ell>0$, then $\cP(j+1,\ell-1)$
holds as well.
\end{prop}

\begin{proof}
By hypothesis, $\UU^\ell(A_j(i))$ is a subspace of $\bR^{n-\ell+1}$ of dimension
at least $j+\ell+1$ for each large enough $i$.  Comparing dimensions yields
$j+2\ell\le n$.  Now, suppose that $\ell>0$, and set $A=A_m(i)$ for a choice
of integers $m\in\{0,\dots,j+1\}$ and $i\ge 0$.  By Corollary \ref{three:cor}(i),
we have
\[
 \min\{\dim\UU^\ell(B), m+\ell\} \le \dim\UU^{\ell-1}(A)
\]
for any subspace $B$ of $A$.  If $m\le j$, we choose $B=A_m(i)$.  Otherwise,
we choose $B=A_{m-1}(i)$.  Then, assuming $i$ large enough, we have
$\dim\UU^\ell(B) \ge \dim(B)+\ell \ge m+\ell$ because of $\cP(j,\ell)$,
and so $\dim\UU^{\ell-1}(A) \ge m+\ell$. Thus $\cP(j+1,\ell-1)$ holds.
\end{proof}

%
%

\section{Property $\cP(0,\ell)$}
\label{sec:P0}

We keep the notation of the preceding section, and fix a real
number $\lambda$ with $0<\lambda<\hlambda_n(\xi)$, thus
$\lambda<1$.  In this section, we derive useful consequences
of the assumption that, for some
$j\ge 0$ and $\ell\ge 1$, we have $\cP(j,\ell-1)$ but not
$\cP(j,\ell)$.  As an example of application, we recover an
important result of Badziahin and Schleischitz from \cite{BS2021}
which yields $\cP(0,\ell)$ under a simple condition on $\ell$.
We recall our convention that all implicit
multiplicative constants are independent of $i$, and we start with
two general lemmas.

\begin{lemma}
\label{P0:lemma:A}
Let $0\le j,\ell\le n$ be integers.  Suppose that
\[
 d:=\liminf_{i\to\infty} \dim\UU^\ell(A_j(i)) \le n-\ell.
\]
Then there are arbitrarily large integers $i\ge 1$ for which
\begin{equation}
 \label{P0:lemma:A:eq}
  \dim\UU^\ell(A_j(i))=d \et \cU^\ell(\ux_{i-1})\not\subseteq \UU^\ell(A_j(i)).
\end{equation}
\end{lemma}

\begin{proof}
Set $V_i=\UU^\ell(A_j(i))$ for each $i\ge 0$, and let $E$ denote the infinite
set of integers $i\ge 0$ for which $\dim(V_i)=d$.  Choose $p\in E$ large enough so that
$\dim(V_i)\ge d$ for each $i\ge p$.   Using \eqref{minimal:sum},
we find
\[
 \sum_{i=p}^\infty V_i
  = \cU^{\ell}(\Span{\ux_p,\ux_{p+1},\dots})
  = \cU^{\ell}(\bR^{n+1})
  = \bR^{n-\ell+1}.
\]
Thus, there exists a smallest $q\in E$ with $q\ge p$ such that
$\sum_{i=p}^q V_i = \bR^{n-\ell+1}$.   Since $V_p$ has dimension
$d\le n-\ell$, it is a proper subspace of $\bR^{n-\ell+1}$.  So, we
must have $q>p$, and thus $V_{q-1}\neq V_{q}$ by the choice of $q$.
As $\dim(V_{q-1})\ge d=\dim(V_{q})$, this means that
$V_{q-1}\not\subset V_{q}$ and therefore
\eqref{P0:lemma:A:eq} holds for $i=q>p$.
\end{proof}

For any integer $m\ge 1$, any subspace of $\bR^m$ defined over
$\bQ$ has height at least $1$.  The next lemma provides an instance
where this lower bound can be improved.

\begin{lemma}
\label{P0:lemma:S}
Let $\ell\in\{0,\dots,n\}$ and let $V$ be a subspace of $\bR^{n+1-\ell}$
defined over $\bQ$.  Suppose that $\cU^\ell(\ux_i)\subseteq V$ and
that $\cU^\ell(\ux_{i-1})\not\subseteq V$ for some $i\ge 1$.  Then,
$1\ll H(V)L_{i-1}$.
\end{lemma}

\begin{proof}
Choose $m\in\{0,\dots,\ell\}$ such that $\uy:=\ux_{i-1}^{(m,\ell)}
\notin V$ and set $U=\langle \uy,\uz\rangle$ where
$\uz:=\ux_{i}^{(m,\ell)}$.  Then, we have $U\cap V=\langle \uz\rangle$
and so $H(U\cap V)=g^{-1}\|\uz\|\asymp g^{-1}X_i$ where $g$
denotes the gcd of the coordinates of $\uz$.  Since $\uy$ and $g^{-1}\uz$
are integer points of $U$, we also find
\[
  H(U)
    \le \|\uy\wedge g^{-1}\uz\|
    =g^{-1}\|\uy\wedge\uz\|
    \ll g^{-1}X_iL_{i-1},
\]
thus $1\le H(U+V) \ll H(U\cap V)^{-1}H(U)H(V)\ll H(V)L_{i-1}$.
\end{proof}

We now come to the main result of this section.

\begin{prop}
\label{P0:prop:I}
Let $j\ge 0$ and $\ell\ge 1$ be integers with $j+2\ell\le n$.  Suppose that
$\cP(j,\ell-1)$ holds but not $\cP(j,\ell)$.  Then, there are arbitrarily large
values of $i\ge 1$ for which
\begin{equation}
\label{P0:prop:I:eq1}
  \dim\UU^\ell(A_j(i)) = \ell+j  \et
  \cU^\ell(\ux_{i-1}) \not\subseteq \UU^\ell(A_j(i)) \varsubsetneq \bR^{n-\ell+1}.
\end{equation}
For those $i$, we further have
\begin{equation}
\label{P0:prop:I:eq2}
  1\ll H\big(\UU^{\ell-1}(A_j(i))\big) L_{i-1}^{n-j-2\ell+2}.
\end{equation}
\end{prop}

\begin{proof}
Since $\cP(j,\ell-1)$ holds, Corollary \ref{three:cor}(ii)
gives
\[
 \dim\UU^\ell(A_j(i)) \ge  \min\{ \ell+j,\, n-\ell+1 \} = \ell+j
\]
for each sufficiently large $i$.  Since $\cP(j,\ell)$ does not hold,
we conclude that
\[
   \liminf_{i\to\infty} \dim\UU^\ell(A_j(i)) =\ell+j \le n-\ell.
\]
Then, Lemma \ref{P0:lemma:A} provides infinitely
many $i$ for which \eqref{P0:prop:I:eq1} holds.
For such $i$, Proposition \ref{three:prop:height} applies
with $A=A_j(i)$, $d=\ell+j$ and $t\in\{\ell-1,\ell\}$.  Setting
$V=\cU^{n-\ell-j}(A)$, it gives
\[
  H(\cU^{\ell-1}(A)) \asymp H(V)^{n-j-2\ell+2}
  \et
  \cU^{n-\ell-j}(\ux_{i-1})\not\subseteq V.
\]
Since $\cU^{n-\ell-j}(\ux_i)\subseteq V$, Lemma \ref{P0:lemma:S}
yields $1\ll H(V)L_{i-1}$. Then \eqref{P0:prop:I:eq2} follows.
\end{proof}

The following restatement of \cite[Lemma 3.1]{BS2021} is
central to the present paper.

\begin{prop}[Badziahin-Schleischitz, 2021]
\label{P0:prop:BS}
Suppose that $\hlambda_n(\xi)>1/(n-\ell+1)$ for some integer $\ell$ with
$0\le \ell\le n/2$.  Then $\cP(0,\ell)$ holds.  More precisely, we have
$\dim\cU^\ell(\ux_i)=\ell+1$ for each sufficiently large $i$ .
\end{prop}

The proof given in \cite{BS2021} is based on a method of Laurent from \cite{La2003}.
To illustrate Proposition \ref{P0:prop:I}, we give the following alternative
argument.

\begin{proof}
Since $\cP(0,0)$ holds, there is a largest integer $m\ge 0$ for which
$\cP(0,m)$ holds.  Suppose that $m<\ell$.  Since $\cP(0,m+1)$ does not hold
and $2(m+1)\le n$, Proposition \ref{P0:prop:I} shows the existence of
arbitrarily large values of $i$ for which
\[
  1 \ll H(\cU^m(\ux_i))L_{i-1}^{n-2m}.
\]
As \eqref{three:eq:HUx} gives $H(\cU^m(\ux_i))\ll X_iL_i^m$,
this yields $1\ll X_iL_{i-1}^{n-m}$.  Then, by the formulas
\eqref{minimal:eq:lambda}, we deduce that $\hlambda_n(\xi)
\le 1/(n-m)\le 1/(n-\ell+1)$.  This contradiction shows that
$\cP(0,\ell)$ holds.  The second assertion of the lemma follows since
$\dim\cU^\ell(\ux_i) \le \ell+1$ for each $i\ge 0$.
\end{proof}

If $\cP(0,\ell)$ holds for some $\ell\ge 1$, then $\dim \UU^\ell(\ux_i)
=\ell+1$ for each large enough $i$ and, for these $i$, we obtain
$1\le H(\UU^\ell(\ux_i))\ll X_iL_i^\ell \ll X_iX_{i+1}^{-\ell\lambda}$
by \eqref{three:eq:HUx} and our choice of $\lambda$.
This implies that $\lambda\le 1/\ell$ and so
$\hlambda_n(\xi)\le 1/\ell$.   Combined with Proposition \ref{P0:prop:I},
this observation yields
\[
 \hlambda_n(\xi)\le \max\{1/(n-\ell+1),1/\ell\}
\]
for each integer $\ell$ with $1\le \ell\le n/2$.  Thus, if $n\ge 2$, we
obtain  $\hlambda_n(\xi)\le 1/\lfloor n/2\rfloor$, which is the estimate
of Davenport and Schmidt from \cite{DS1969} mentioned in the introduction.
We conclude with another consequence of $\cP(0,\ell)$.

\begin{lemma}
\label{P0:lemma:Fbis}
Suppose that $\cP(0,\ell)$ holds for some integer $\ell\ge 1$ with $2\ell<n$.  Then,
there are infinitely many $i\ge 1$ for which $\cU^\ell(\ux_{i-1})\not\subseteq
\cU^\ell(\ux_i)$ and, for those $i$, we have
\begin{equation*}
\label{P0:lemma:Fbis:eq}
  X_{i+1}^\theta \ll X_i
     \quad\text{where}\quad
  \theta=\frac{\ell\lambda}{1-\lambda}
\end{equation*}
\end{lemma}

\begin{proof}
Since $\cP(0,\ell)$ holds, we have $\dim\cU^\ell(\ux_i)=\ell+1\le n-\ell$
for each sufficiently large $i$.  By Lemma \ref{P0:lemma:A}, we deduce that
both $\dim\cU^\ell(\ux_i)=\ell+1$ and $\cU^\ell(\ux_{i-1})\not\subseteq
\cU^\ell(\ux_i)$ for infinitely many $i\ge 1$.  For those $i$, we have
$H(\cU^\ell(\ux_i))\ll X_iL_i^\ell$ by \eqref{three:eq:HUx}, and Lemma
\ref{P0:lemma:S} gives
\[
  1 \ll H(\cU^\ell(\ux_i))L_{i-1} \ll X_iL_i^\ell L_{i-1}
     \ll X_i^{1-\lambda}X_{i+1}^{-\ell\lambda},
\]
so $X_{i+1}^\theta \ll X_i$.
\end{proof}

%
%

\section{Property $\cP(1,\ell)$}
\label{sec:P1}

With the notation of the preceding sections (including the choice of
$\lambda$), we provide below a sufficient condition on $\ell$ and $\lambda$
for property $\cP(1,\ell)$ to hold.  We also establish consequences of
$\cP(1,\ell)$.  The results of this section provide a first step towards the
proof of Theorems \ref{intro:thm:impair} and \ref{intro:thm:pair}.  We start
with a crude height estimate.

\begin{lemma}
\label{P1:lemma1}
If $\cP(1,\ell)$ holds for some $\ell\ge 0$, then
$H(\UU^\ell(\ux_i,\ux_{i+1}))\ll X_{i+1}^{1-(\ell+1)\lambda}$.
\end{lemma}

\begin{proof}
For given $i\ge 0$ and $\ell\in\{0,\dots,n\}$, the subspace
$\UU^\ell(\ux_i,\ux_{i+1})$ of $\bR^{n-\ell+1}$ is generated by
the set $\{\ux_i^{(0,\ell)},\dots,\ux_i^{(\ell,\ell)},
\ux_{i+1}^{(0,\ell)},\dots,\ux_{i+1}^{(\ell,\ell)}\}$ which consists
of integer points $\uy$ with $\norm{\uy}\le X_{i+1}$ and
$L_\xi(\uy)\ll L_i$.  If $\cP(1,\ell)$ holds
and $i$ is large enough, this space has dimension $d\ge \ell+2$,
and so Lemma \ref{heights:lemma} gives $H(\UU^\ell(\ux_i,\ux_{i+1}))
\ll X_{i+1}L_i^{d-1}\le X_{i+1}^{1-(\ell+1)\lambda}$.
\end{proof}

By the above, property $\cP(1,\ell)$ implies that
$\lambda\le 1/(\ell+1)$ and so $\hlambda_n(\xi)\le 1/(\ell+1)$.  The
next lemma provides finer estimates.

\begin{lemma}
\label{P1:lemma2}
Suppose that $n\ge 2$ and that $\cP(1,\ell)$ holds for some integer
$\ell\ge 0$.  Then, for each pair of consecutive elements $i<j$ of $I$,
we have
\[
 H(\UU^\ell(\ux_i,\ux_{i+1}))
    \ll X_{j+1}^{-\ell\lambda}X_{i+1}^{1-\lambda}
 \et
 H(\UU^\ell(\ux_{i-1},\ux_i,\ux_{i+1}))
    \ll X_{j+1}^{-\ell\lambda}X_{i+1}^{1-\lambda}X_i^{-e\lambda},
\]
where $e=\dim\UU^\ell(\ux_{i-1},\ux_i,\ux_{i+1})-\ell-2$.
\end{lemma}

\begin{proof}
For a pair of consecutive elements $i<j$ of $I$, we have $\Span{\ux_i,\ux_{i+1}}=
\Span{\ux_{j-1},\ux_j}$.  Thus we have a chain of subspaces
\begin{align*}
 U=\UU^\ell(\ux_j)
  &\subseteq
  V=\UU^\ell(\ux_i,\ux_{i+1})=U+\UU^\ell(\ux_{j-1}) \\
  &\subseteq
  W=\UU^\ell(\ux_{i-1},\ux_i,\ux_{i+1})=V+\UU^\ell(\ux_{i-1}).
\end{align*}
If $i$ is large enough, then by property $\cP(1,\ell)$ we have $\dim(U)=\ell+1$,
$\dim(V)=\dim(U)+a$ and $\dim(W)=\dim(V)+b$ for some integers $a\ge 1$
and $b\ge0$.  As each subspace $\UU^\ell(\ux_h)$ is generated by integer
points $\uy$ with $\norm{\uy}\le X_h$ and $L_\xi(\uy)\ll L_h$, Lemma
\ref{heights:lemma} gives
\[
 H(V) \ll X_j L_j^\ell L_{j-1}^a
 \et
 H(W) \ll X_j L_j^\ell L_{j-1}^a L_{i-1}^b.
\]
Finally, by Lemma \ref{minimal:XLXL}, we have $X_jL_{j-1}\asymp X_{i+1}L_i$.
Since $a\ge 1$, we deduce that
\[
 H(V) \ll L_j^\ell X_j L_{j-1} \asymp L_j^\ell X_{i+1}L_i
 \et
 H(W) \ll L_j^\ell X_j L_{j-1} L_{i-1}^e \asymp L_j^\ell X_{i+1}L_i L_{i-1}^e,
\]
where $e=a+b-1=\dim(W)-\ell-2\ge 0$.  The conclusion follows.
\end{proof}

Under the hypotheses of Lemma \ref{P1:lemma2}, we have
$1\le H(\UU^\ell(\ux_i,\ux_{i+1}))  \ll X_{j+1}^{-\ell\lambda}X_{i+1}^{1-\lambda}$
and so $X_{j+1}^\theta\ll X_{i+1}$ with $\theta=\ell\lambda/(1-\lambda)$,
for each pair of consecutive elements $i<j$ of $I$.  A fortiori, this implies that
$X_{i+1}^\theta\ll X_{i}$ for each $i\ge 0$.  The following result yields a weaker estimate but assumes $\cP(1,\ell-1)$ instead of $\cP(1,\ell)$.

\begin{lemma}
\label{P1:lemma3}
Suppose that $\cP(1,\ell-1)$ holds for some integer $\ell$ with
$1\le \ell\le n/2$, and that $\hlambda_n(\xi) > 1/(2\ell)$.
Then, for each $i\ge 0$, we have $X_iL_i^\ell\gg 1$ and so
$X_{i+1}^{\ell\lambda}\ll X_i$ .
\end{lemma}

\begin{proof}
By definition of property $\cP(1,\ell-1)$ we have,
\[
  \dim\UU^{\ell-1}(\ux_{i-1},\ux_i)\ge \ell+1
  \et
  \dim\UU^{\ell-1}(\ux_i)\ge \ell
\]
for each sufficiently large $i\ge 1$.  Fix such an integer $i$.  Then,
Corollary \ref{three:cor}(ii) gives
\[
  \dim\UU^{\ell}(\ux_i)\ge \min\{\dim\UU^{\ell-1}(\ux_i),n-\ell+1\}\ge \ell.
\]
If $\dim\UU^{\ell}(\ux_i)\ge \ell+1$, then $1\le H(\UU^{\ell}(\ux_i)) \ll X_iL_i^{\ell}$
and we are done.  Otherwise, $\UU^\ell(\ux_i)$ has dimension $\ell$. Then,
Proposition \ref{three:prop:height} applies with $A=\Span{\ux_i}$,
$j=0$, $d=\ell$, $V=\UU^{n-\ell}(\ux_i)$  and $t=\ell-1$.  It gives
\[
  \dim\UU^{\ell-1}(\ux_i)=\ell
  \et
  H(\UU^{\ell-1}(\ux_i)) \asymp H(V)^{n-2\ell+2} \ge H(V)^2 .
\]
Moreover, since $\dim\UU^{\ell-1}(\ux_{i-1},\ux_i)>\ell$, we have
$\UU^{\ell-1}(\ux_{i-1}) \not\subseteq\UU^{\ell-1}(\ux_i)$ and so the
proposition also gives $\UU^{n-\ell}(\ux_{i-1}) \not\subseteq V$.
By Lemma \ref{P0:lemma:S}, this in turn yields $H(V)L_{i-1}\gg 1$, thus
we find
\[
 1\ll H(\UU^{\ell-1}(\ux_i))L_{i-1}^2 \ll X_iL_i^{\ell-1}L_{i-1}^2 \ll X_i^{1-2\lambda}L_i^{\ell-1}.
\]
As $\hlambda_n(\xi) > 1/(2\ell)$, we may assume that $\lambda\ge 1/(2\ell)$.
This gives $1\ll X_i^{1-1/\ell}L_i^{\ell-1}$ and so $1\ll X_iL_i^\ell$.
\end{proof}

We now come to the main result of this section.

\begin{prop}
\label{P1:prop}
Suppose that, for some integer $\ell\ge 1$ with $1+2\ell\le n$, property $\cP(0,\ell)$
holds but not $\cP(1,\ell)$.  Then we have $0\le 1-(n-\ell)\lambda-\ell\lambda^2$.
\end{prop}

\begin{proof}
Since $\cP(0,\ell)$ holds, we have $\ell+1=\dim\UU^\ell(A_0(i))
\le \dim\UU^\ell(A_1(i))$ for each sufficiently large $i\ge 0$.  Since $\cP(1,\ell)$
does not hold, we also have $\dim\UU^\ell(A_1(i))\le \ell+1$ for arbitrarily large
values of $i$.  Thus, Lemma \ref{P0:lemma:A} applies with $j=1$ and $d=\ell+1$,
and so there are arbitrarily large integers $i\ge 1$ with
\[
 \dim\UU^\ell(A_1(i))=\ell+1
 \et
 \UU^\ell(\ux_{i-1})\not\subseteq \UU^\ell(A_1(i))\subseteq\bR^{n-\ell+1}.
\]
For these $i$, Proposition \ref{three:prop:height} applies with $A=A_1(i)$,
$j=1$, $d=\ell+1$ and any $t\in\{\ell-1,\ell\}$.  Setting
$V=\UU^{n-\ell-1}(A_1(i))$, it gives
\[
 H\big(\UU^{\ell-1}(A_1(i))\big)\asymp H(V)^{n-2\ell+1}
 \et
 H\big(\UU^{\ell}(A_1(i))\big)\asymp H(V)^{n-2\ell}.
\]
Since  $\UU^\ell(\ux_{i-1})\not\subseteq \UU^\ell(A_1(i))$, it also gives
$\UU^{n-\ell-1}(\ux_{i-1})\not\subseteq V$.  By Lemma \ref{P0:lemma:S},
this implies that $1\ll H(V)L_{i-1}\ll H(V)X_i^{-\lambda}$ and so
\[
 X_i^\lambda \ll H(V).
\]
Because of $\cP(0,\ell)$, the subspace $\UU^\ell(\ux_i)$ of
$\UU^\ell(A_1(i))$ has dimension $\ell+1$ when $i$ is large enough, and
then it coincides with $\UU^\ell(A_1(i))$.  Moreover, $\cP(1,\ell-1)$ holds
by Proposition \ref{minimal:prop:Pjl}. Thus using \eqref{three:eq:HUx} and
Lemma \ref{P1:lemma1}, the above estimates imply
\begin{align*}
 X_i^{(n-2\ell)\lambda}
 &\ll H(\UU^\ell(\ux_i))\ll X_iL_i^\ell\ll X_iX_{i+1}^{-\ell\lambda},\\
 X_i^{(n-2\ell+1)\lambda}
 &\ll H(\UU^{\ell-1}(\ux_i,\ux_{i+1}))\ll X_{i+1}^{1-\ell\lambda}.
\end{align*}
The second row of estimates implies $1-\ell\lambda>0$ and provides a
lower bound for $X_{i+1}$ in terms of $X_i$.  Substituting it in the first row
and comparing powers of $X_i$, we deduce that
\[
 (1-\ell\lambda)(n-2\ell)\lambda
\le (1-\ell\lambda)-(\ell\lambda)(n-2\ell+1)\lambda,
\]
which after simplications reduces to $0\le 1-(n-\ell)\lambda-\ell\lambda^2$.
\end{proof}

\begin{cor}
\label{P1:cor}
Suppose that $n\ge 3$.  Let $\ell$ be an integer with $1\le \ell < n/2$
and let $\rho$ denote the unique positive root of the polynomial
$P(x)=1-(n-\ell)x-\ell x^2$.  Then, we have
$\rho>1/(n-\ell+1)$.  If $\hlambda_n(\xi)>\rho$, then $\cP(1,\ell)$ holds.
\end{cor}

\begin{proof}
Let $k=n-\ell+1$.  Since $P(1/k)=(n-2\ell+1)/k^2>0$, we have $\rho>1/k$.
If $\hlambda_n(\xi)>\rho$, we may assume that $\lambda>\rho$.  Then
we have $\lambda>1/k=1/(n-\ell+1)$ and thus property $\cP(0,\ell)$ holds
by Proposition \ref{P0:prop:BS}.  Since $\lambda>\rho$, we also have
$P(\lambda)<0$, and so the preceding proposition implies that $\cP(1,\ell)$
holds.
\end{proof}

We remarked after Lemma \ref{P1:lemma1} that property $\cP(1,\ell)$ implies
$\hlambda_n(\xi)\le1/(\ell+1)$.  Thus, with the notation and hypotheses of the
above corollary, we obtain
\[
 \hlambda_n(\xi)\le \max\{1/(\ell+1),\rho\}.
\]
If $n=2m+1\ge 3$ is odd, we may choose $\ell=m$.  Then $\rho$ is the positive
root of $P(x)=1-(m+1)x-mx^2$, denoted by $\alpha_m$ in the statement of
Theorem \ref{intro:thm:impair}.  As $P(1/(m+1))<0$, this yields
$\hlambda_n(\xi)\le 1/(m+1)=\lceil n/2\rceil^{-1}$ which is the main
result of Laurent in \cite{La2003}.

%
%

\section{First general height estimates}
\label{sec:first}

With the notation of the preceding section (including the choice of
$\lambda$), we first show that property $\cP(j,\ell)$ yields
special bases for the subspaces $\UU^\ell(A_j(i))$.  Then, we
deduce an upper bound on the height of these subspaces in terms of
the quantities $Y_j(i)$ from Definition \ref{minimal:def:AY}.

\begin{lemma}
\label{first:lemma:basis}
Suppose that $\cP(j,\ell)$ holds for some integers $j,\ell\in\{0,\dots,n\}$.
For each large enough integer $i\ge 0$ and each integer $q\ge i$ such that
$A_j(i)=\langle \ux_i,\dots,\ux_q\rangle$, there exist an integer $e\ge 0$
and a basis $\{\uy_0,\uy_1,\dots,\uy_{\ell+j+e}\}$ of $\UU^\ell(A_j(i))$
made of points of $\bZ^{n+1-\ell}$ of norm $\le X_q$ with
\[
\begin{cases}
  L_\xi(\uy_m)\ll X_{q+1}^{-\lambda} &\text{for \ $0\le m\le \ell$,}\\
  L_\xi(\uy_{\ell+m})\ll Y_{j-m}(i)^{-\lambda} &\text{for \ $1\le m\le j$,}\\
  L_\xi(\uy_{\ell+j+m})\ll Y_{0}(i)^{-\lambda} &\text{for \ $1\le m\le e$.}
\end{cases}
\]
\end{lemma}

\begin{proof}
We proceed by induction on $j$.  If $j=0$, we have $q=i$, and
$\UU^{\ell}(A_0(i))=\cU^\ell(\ux_i)$ has dimension $\ell+1$ for $i$
large enough.  Then the points $\uy_m=\ux_i^{(m,\ell)}$ for $m=0,\dots,\ell$
have the required properties.  Now suppose that $j\ge 1$.  For $i$ large enough,
$\UU^\ell(A_j(i))$ has dimension $\ge \ell+j+1$.   Choose $q\ge i$ such that
$A_j(i)=\langle \ux_i,\dots,\ux_q\rangle$, and choose $p$ minimal with
$i<p\le q$ such that $\dim\langle \ux_p,\dots,\ux_q\rangle=j$ or equivalently
such that $A_{j-1}(p)=\langle \ux_p,\dots,\ux_q\rangle$.  Since
$\cP(j-1,\ell)$ holds, we may assume by induction that,
when $i$ is large enough, the vector space
$\UU^\ell(A_{j-1}(p))$ contains linearly independent points
$\uy_0,\uy_1,\dots,\uy_{\ell+j-1}$ of $\bZ^{n+1-\ell}$ of norm $\le X_q$
with $L_\xi(\uy_m)\ll X_{q+1}^{-\lambda}$ for $0\le m\le \ell$ and
$L_\xi(\uy_{\ell+m})\ll Y_{j-1-m}(p)^{-\lambda}$ for $1\le m\le j-1$.
Since $\ux_{p-1}\notin\langle\ux_p,\dots,\ux_q\rangle$, we have
\[
 \dim\langle\ux_i,\dots,\ux_r\rangle
  \ge 1+\dim\langle\ux_p,\dots,\ux_r\rangle
\]
for each $r=p,\dots,q$.  Thus, for each integer $m$ with $1\le m\le j-1$,
we have $\sigma_m(i) \le \sigma_{m-1}(p) < q$, and so $Y_m(i)\le Y_{m-1}(p)$.
By the above, this means that $L_\xi(\uy_{\ell+m})\ll Y_{j-m}(i)^{-\lambda}$
for $1\le m\le j-1$.  Since $\{\uy_0,\uy_1,\dots,\uy_{\ell+j-1}\}$
is a linearly independent subset of $\UU^\ell(A_j(i))$, we may complete
it to a basis $\{\uy_0,\uy_1,\dots,\uy_{\ell+j+e}\}$ for some $e\ge 0$
by adding $e+1$ points of the form $\ux_h^{(s,\ell)}$ with $i\le h\le q$
and $0\le s\le \ell$.  These new points belong to $\bZ^{n+1-\ell}$,
have norm $\le X_q$, and satisfy $L_\xi(\uy_{\ell+j+m})\ll L_i
\ll Y_{0}(i)^{-\lambda}$ for $0\le m\le e$.
\end{proof}

\begin{prop}
\label{first:prop}
Suppose that $\cP(j,\ell)$ holds for some integers $j,\ell\in\{0,\dots,n\}$
with $j\ge 1$.  For each large enough $i\ge 0$, we have
\[
 H\big(\UU^\ell(A_j(i))\big)
   \ll Y_{j-1}(i)^{1-\ell\lambda}
       \Big( \prod_{m=1}^{j} Y_{j-m}(i)^{-\lambda} \Big)
        Y_0(i)^{-e\lambda}
\]
with $e=\dim\UU^\ell(A_j(i))-\ell-j-1\ge 0$.
\end{prop}

\begin{proof}
For given $i\ge 0$, choose $q\ge i$ minimal such that
$A_j(i)=\langle\ux_i,\dots,\ux_q\rangle$.  Then, assuming $i$ large
enough so that $e\ge 0$, consider the basis $\{\uy_0,\uy_1,\dots,\uy_{\ell+j+e}\}$
of $\UU^\ell(A_j(i))$ provided by Lemma \ref{first:lemma:basis}.
By the choice of $q$, we have $X_q=Y_{j-1}(i)\le X_{q+1}$ and so
$L_\xi(\uy_m)\ll Y_{j-1}(i)^{-\lambda}$ for $0\le m\le \ell$.
Since this basis consists of integer points, we also have
\[
  H\big(\UU^\ell(A_j(i))\big)
    \le \|\uy_0\wedge\uy_1\wedge\cdots\wedge\uy_{\ell+j+e}\|.
\]
We conclude by applying Lemma \ref{heights:lemma} along with
the estimates of Lemma \ref{first:lemma:basis}.
\end{proof}

When $\ell=0$, property $\cP(j,\ell)$ holds and we obtain the following
estimate.

\begin{corollary}
  \label{first:cor:HA}
 $\disp H(A_j(i)) \ll Y_{j-1}(i) \prod_{m=0}^{j-1} Y_m(i)^{-\lambda}$
 for all $j=0,\dots,n-1$ and all $i\ge 0$.
\end{corollary}

\begin{cor}
\label{first:cor:T}
Suppose that $\cP(j,\ell)$ holds for some integers $j\ge 1$ and $\ell\ge 0$ with
$j+2\ell<n$.  Then there are arbitrarily large integers $i\ge 1$ for which
$\UU^\ell(\ux_{i-1})\not\subseteq\UU^\ell(\ux_i)$ and
\begin{equation}
\label{first:cor:T:eq}
  1\ll Y_{j-1}(i)^{1-\ell\lambda}
       \Big(\prod_{m=1}^{j}Y_{j-m}(i)^{-\lambda}\Big)
       X_i^{-\lambda}.
\end{equation}
\end{cor}

\begin{proof}
Let $d=\liminf_{i\to\infty}\dim\UU^\ell(A_j(i))$.  If $d=\ell+j+1$, then we
have $d\le n-\ell$ and Lemma \ref{P0:lemma:A} provides  infinitely many
$i\ge 1$ for which $\cU^\ell(\ux_{i-1})\not\subseteq \UU^\ell(A_j(i))$.
For those $i$, we have $\UU^\ell(\ux_{i-1})\not\subseteq\UU^\ell(\ux_i)$
and Lemma \ref{P0:lemma:S} gives $1\ll H(V)L_{i-1}$
with $V=\UU^\ell(A_j(i))$. Then \eqref{first:cor:T:eq} follows using
$L_{i-1}\ll X_i^{-\lambda}$ and the upper bound for $H(V)$ provided by
Proposition \ref{first:prop}.  Otherwise, for each sufficiently large $i$,
we have $\dim\UU^\ell(A_j(i)) > \ell+j+1$  and \eqref{first:cor:T:eq}
follows directly from Proposition \ref{first:prop} using
$H\big(\UU^\ell(A_j(i))\big)\ge 1$, $Y_0(i)=X_{i+1}\ge X_i$ and
$e\ge 1$.  Moreover, as $2\ell<n$, Lemma \ref{P0:lemma:A} also gives
$\UU^\ell(\ux_{i-1})\not\subseteq\UU^\ell(\ux_i)$ for infinitely many $i\ge 1$.
\end{proof}

%
%

\section{An alternative height estimate}
\label{sec:alt}

Keeping the same notation, we derive a second height estimate for $\UU^\ell(A_j(i))$ by
an indirect process, as in the proof of Lemma \ref{P0:lemma:S},
namely by writing this space as a sum of two subspaces with a well-chosen one
dimensional intersection, and then by applying Schmidt's height inequality
\eqref{heights:eq:schmidt}.

\begin{prop}
\label{alt:prop}
Suppose that $\cP(j,\ell)$ holds for some integers $1\le j\le \ell<n$.
For each $i\ge 0$, we have
\[
 H\big(\UU^\ell(A_j(i))\big)
   \ll H\big(A_j(i)\big) Y_{j}(i)^{-(\ell-j+1)\lambda}
       \prod_{m=1}^{j-1} Y_m(i)^{-\lambda}.
\]
\end{prop}

The main feature of this estimate is that it involves a negative power of $Y_j(i)$
and so, as we will see, it yields an upper bound for $Y_j(i)$ in terms of
$Y_0(i),\dots,Y_{j-1}(i)$.   The proof requires the following simple observation.

\begin{lemma}
\label{alt:lemma}
Let $\tau_\ua\colon\bR^{n+1}\to\bR^{n-\ell+1}$ be the linear map
given by \eqref{three:eq:tau} for fixed $\ell\in\{0,\dots,n\}$
and $\ua=(a_0,\dots,a_\ell)\in \bZ^{\ell+1}\setminus\{0\}$.
Then we have $\norm{\tau_\ua(\ux_i)}\asymp X_i$ as $i\to\infty$.
\end{lemma}

\begin{proof}
Since $\lim_{i\to\infty} X_i^{-1}\ux_i =\norm{\Xi}^{-1}\Xi$ where
$\Xi=(1,\xi,\dots,\xi^n)$, we find
\[
 \lim_{i\to\infty} X_i^{-1}\tau_\ua(\ux_i)
    =\norm{\Xi}^{-1}\tau_\ua(\Xi)
    =\norm{\Xi}^{-1}(a_0+a_1\xi+\cdots+a_\ell\xi^\ell)(1,\xi,\dots,\xi^{n-\ell}).
\]
As $[\bQ(\xi):\bQ]>n\ge \ell$,  this limit is non-zero, and the conclusion follows.
\end{proof}

\begin{proof}[Proof of Proposition \ref{alt:prop}]
We may assume that $i$ is large enough so that Lemma \ref{first:lemma:basis}
applies.  Choose $q\ge i$ maximal such that
$A_j(i)=\langle\ux_i,\dots,\ux_q\rangle$, and to simplify notation set
$A=A_j(i)$.  Then, by definition, we have $X_{q+1}=Y_j(i)$.  Moreover,
we have $j+2\ell\le n$ by $\cP(j,\ell)$ (see Proposition \ref{minimal:prop:Pjl}),
thus $\dim(A) =j+1\le n$, and so $A$ is a proper subspace of $\bR^{n+1}$.
Using the basis of $\cU^\ell(A)$ provided by Lemma \ref{first:lemma:basis}
for the present choice of $q$, we set
\[
 W=\langle \uy_0,\uy_1,\dots,\uy_{\ell+j-1}\rangle
   \subseteq \bR^{n+1-\ell}.
\]
Since $\dim(W)=\ell+j<\dim\cU^\ell(A)$, we have $\cU^\ell(A)\not\subseteq W$.
As $A$ is a proper subspace of $\bR^{n+1}$, Proposition \ref{three:prop:avoiding}
provides a non-zero
point $\ua=(a_0,\dots,a_\ell)\in\bZ^{\ell+1}$ with $\|\ua\|\ll 1$ such that
the linear map $\tau_\ua\colon\bR^{n+1}\to\bR^{n+1-\ell}$ is injective
on $A$ with $\tau_\ua(A)\not\subseteq W$. Thus we have
$\dim(\tau_\ua(A)\cap W)  \le j$ and so there are $\ell$ points
$\uz_0,\dots,\uz_{\ell-1}$ among $\uy_0,\uy_1,\dots,\uy_{\ell+j-1}$ such
that
\[
  \tau_\ua(A) \cap \langle \uz_0,\dots,\uz_{\ell-1} \rangle = 0.
\]
By construction, these are integer points of norm $\le X_q$ and, since $j\le \ell$,
we may order them so that
\begin{equation}
 \label{lemma:D:eq1}
\begin{cases}
  L_\xi(\uz_m)\ll X_{q+1}^{-\lambda}=Y_{j}(i)^{-\lambda}
       &\text{for $0\le m\le \ell-j$,}\\
  L_\xi(\uz_{\ell-j+m})\ll Y_{j-m}(i)^{-\lambda}
       &\text{for $1\le m\le j-1$.}
\end{cases}
\end{equation}
Since $\dim \cU^\ell(A) =\ell+j+1+e$ for some integer $e\ge 0$,
we may complete $\{\uz_0,\dots,\uz_{\ell-1}\}$ to a maximal linearly
independent subset $\{\uz_0,\dots,\uz_{\ell-1+e}\}$ of $\cU^\ell(A)$
such that
\[
  \cU^\ell(A)
  = \tau_\ua(A) \oplus \langle \uz_0,\dots,\uz_{\ell-1+e} \rangle
\]
by adding integer points of the form $\ux_h^{(s,\ell)}$ with $i\le h\le q$
and $0\le s\le \ell$.  These new points have norm $\le X_q$ and satisfy
\begin{equation}
 \label{lemma:D:eq2}
  L_\xi(\uz_{\ell-1+m}) \ll L_i \ll Y_0(i)^{-\lambda}
   \quad \text{for $1\le m\le e$.}
\end{equation}
Define
\[
  U=\tau_\ua(A), \quad \uz=\tau_\ua(\ux_q)
  \et V=\langle \uz,\uz_0,\dots,\uz_{\ell-1+e} \rangle,
\]
so that
\[
 U+V=\cU^\ell(A)
 \et
 U\cap V=\Span{\uz}.
\]
Since $\norm{\ua}\ll 1$, we find
\[
  H(U)\ll H(A), \quad \norm{\uz}\asymp X_q
  \et  L_\xi(\uz) \ll L_q\ll X_{q+1}^{-\lambda}=Y_j(i)^{-\lambda},
\]
where the middle estimate $\norm{\uz}\asymp X_q$ comes from
Lemma \ref{alt:lemma}.  We deduce that
\[
 H(U\cap V)=g^{-1}\|\uz\|\asymp g^{-1}X_q
\]
where $g$ denotes the content of $\uz$.  Since $g^{-1}\uz$ is an integer point,
we further have
\[
  H(V) \le \|g^{-1}\uz\wedge\uz_0\wedge\cdots\wedge\uz_{\ell-1+e}\|
      = g^{-1}\|\uz\wedge\uz_0\wedge\cdots\wedge\uz_{\ell-1+e}\|.
\]
Applying Lemma \ref{heights:lemma} with the estimates \eqref{lemma:D:eq1},
\eqref{lemma:D:eq2}, this implies
\begin{align*}
  H(V)
    &\ll g^{-1} X_q Y_j(i)^{-(\ell-j+1)\lambda}
       \Big(\prod_{m=1}^{j-1} Y_{j-m}(i)^{-\lambda}\Big) Y_0(i)^{-e\lambda}\\
    &\ll H(U\cap V) Y_j(i)^{-(\ell-j+1)\lambda}
       \prod_{m=1}^{j-1} Y_{j-m}(i)^{-\lambda}.
\end{align*}
The conclusion follows since $H(\cU^\ell(A))\ll H(U)H(V)/H(U\cap V)$
by \eqref{heights:eq:schmidt}.
\end{proof}

Combining Proposition \ref{alt:prop} with the crude estimate
$H\big(\UU^\ell(A_j(i))\big)\ge 1$ and the upper bound for $H(A_j(i))$
given by Corollary \ref{first:cor:HA}, we obtain the following upper bound
for $Y_j(i)$.

\begin{corollary}
\label{alt:cor:E}
Suppose that $\cP(j,\ell)$ holds for some integers $1\le j\le \ell<n$.
Then, for each $i\ge 0$, we have
\[
  Y_{j}(i)^{(\ell-j+1)\lambda}
      \ll Y_{j-1}(i) \Big(\prod_{m=1}^{j-1} Y_m(i)^{-2\lambda} \Big)
          Y_0(i)^{-\lambda}.
\]
\end{corollary}

When $j=1$, this can be reformulated as follows.

\begin{corollary}
\label{alt:cor:F}
Suppose that $\cP(1,\ell)$ holds for some integer $1\le \ell<n$.
Then the ratio
\begin{equation}
  \label{eq:theta}
  \theta=\frac{\ell\lambda}{1-\lambda}
\end{equation}
satisfies $0<\theta\le 1$ and we have $Y_1(i)^\theta\ll Y_0(i)$ and
$Y_0(i)^\theta\ll Y_{-1}(i)$ for each $i\ge 0$.
\end{corollary}

\begin{proof}
By Corollary \ref{alt:cor:E}, we have $Y_1(i)^{\ell\lambda}\ll
Y_0(i)^{1-\lambda}$ for each $i\ge 0$.  Since $\lambda<1$,
this yields $Y_1(i)^\theta\ll Y_0(i)$ for all $i\ge 0$ and thus
$\theta\le 1$.  For $i\ge 1$, this in turn gives
\[
  Y_0(i)^\theta=X_{i+1}^\theta\le Y_1(i-1)^\theta
    \ll Y_0(i-1)=X_i=Y_{-1}(i).
\qedhere
\]
\end{proof}

More generally, Corollary \ref{alt:cor:E} admits the following consequence.

\begin{cor}
\label{alt:cor:G}
Suppose that $\cP(j,\ell)$ holds for some integers $1\le j\le \ell<n$,
and that $\theta^{j-1}+\theta^j\ge 1$ where $\theta$ is given by \eqref{eq:theta}.
Then we have $Y_m(i)^\theta\ll Y_{m-1}(i)$ for each each $i\ge 0$ and
each $m=0,1,\dots,j$.
\end{cor}

\begin{proof}
Since $\cP(1,\ell)$ holds, Corollary \ref{alt:cor:F} gives $0<\theta\le 1$
and $Y_m(i)^\theta\ll Y_{m-1}(i)$ for $m=0,1$.  So, we are done if $j=1$.
Suppose now that $j\ge 2$.  Since $\theta\le 1$, the hypothesis
$\theta^{j-1}+\theta^j\ge 1$ implies that $\theta^{j-2}+\theta^{j-1}\ge 1$.
Since $\cP(j-1,\ell)$ holds, we may
assume by induction that $Y_m(i)^\theta\ll Y_{m-1}(i)$ for each $i\ge 0$
and each $m=0,\dots,j-1$.  Then, using Corollary \ref{alt:cor:E},
we obtain $Y_j(i)^{(\ell-j+1)\lambda} \ll Y_{j-1}(i)^\rho$ where
\begin{align*}
  \rho
   &= 1 - 2 \lambda\Big(\sum_{m=1}^{j-1} \theta^{j-1-m}\Big)
                  - \lambda\theta^{j-1}\\
   &= 1 - \lambda - \lambda \sum_{m=1}^{j-1} (\theta^{m-1}+\theta^m)
     \le \frac{\ell\lambda}{\theta} - \lambda \sum_{m=1}^{j-1} \frac{1}{\theta}
     = \frac{(\ell-j+1)\lambda}{\theta},
\end{align*}
and so $Y_j(i)^\theta \ll Y_{j-1}(i)$.
\end{proof}

%
%

\section{Main proposition and proof of Theorem \ref{intro:thm:main}}
\label{sec:main}

We first prove the following general statement and then optimize the choice of parameters
to deduce Theorem \ref{intro:thm:main}.  The notation, including the choice of $\lambda$
is as in the preceding sections.

\begin{prop}
\label{main:prop}
Let $1\le k\le \ell$ be integers with $k+2\ell=n$.  Suppose that
\[
   \frac{1}{2}\le \theta^k \le 1
     \quad\text{where}\quad
  \theta=\frac{\ell\lambda}{1-\lambda}.
\]
Then we have
$\lambda \le \eta^{-1}$ where $\eta=\ell+\theta+\theta^2+\cdots+\theta^{k+1}$.
\end{prop}

\begin{proof}
Assume on the contrary that $\lambda>\eta^{-1}$.  Since $\theta\le 1$,
we have $\eta\le \ell+k+1=n-\ell+1$, thus $\lambda>(n-\ell+1)^{-1}$,
and so Proposition \ref{P0:prop:BS} shows that
$\cP(0,\ell)$ holds.  By Lemma \ref{P0:lemma:Fbis}, this implies that
$\cU^\ell(\ux_{i-1})\not\subseteq \cU^\ell(\ux_i)$ for infinitely many
$i\ge 1$, and that
\begin{equation}
  \label{main:prop:eq1}
  Y_0(i)^\theta=X_{i+1}^\theta\ll Y_{-1}(i)=X_i
  \quad
  \text{whenever}
  \quad
  \cU^\ell(\ux_{i-1})\not\subseteq \cU^\ell(\ux_i).
\end{equation}

Let $j$ be the largest integer with $0\le j\le k$ for which $\cP(j,\ell)$ holds.
Since $\theta\le 1$, we have $\theta^{j-1}+\theta^j
\ge \theta^{k-1}+\theta^k\ge 1$.    So, if $j\ge 1$, Corollary \ref{alt:cor:G}
gives
\begin{equation}
  \label{main:prop:eq2}
  Y_m(i)^\theta\ll Y_{m-1}(i)
  \quad\text{for each}\quad m=0,1,\dots,j\et i\ge 0.
\end{equation}

Suppose that $j<k$.  Then $\cP(j+1,\ell)$ does not hold.  However,
by Proposition \ref{minimal:prop:Pjl}, $\cP(j+1,\ell-1)$ holds
because $\cP(j,\ell)$ does.  Thus, by Proposition \ref{P0:prop:I},
there are arbitrarily large $i\ge 1$ for which
\[
  \cU^\ell(\ux_{i-1})\not\subseteq \UU^\ell(A_{j+1}(i))
  \et
  1 \ll H\big(\UU^{\ell-1}(A_{j+1}(i))\big)L_{i-1}^{n-j-2\ell+1}.
\]
Using Proposition \ref{first:prop}  to estimate from above the height of
$\UU^{\ell-1}(A_{j+1}(i))$ and recalling that $n=k+2\ell$ and
$L_{i-1}\ll Y_{-1}(i)^{-\lambda}$, this gives
\[
  1 \ll Y_{j}(i)^{1-(\ell-1)\lambda}
       \Big( \prod_{m=0}^{j} Y_{j-m}(i)^{-\lambda} \Big)
        Y_{-1}(i)^{-(k-j+1)\lambda}.
\]
For those $i$, we have $\cU^\ell(\ux_{i-1})\not\subseteq \cU^\ell(\ux_i)$.
Thus, using \eqref{main:prop:eq1} if $j=0$ and \eqref{main:prop:eq2} else, we find
\begin{align*}
  0 &\le 1 - \lambda(\ell-1) - \lambda\Big(\sum_{m=0}^{j} \theta^{m}\Big)
             - \lambda(k-j+1)\theta^{j+1} \\
     &\le 1 - \lambda\Big(\ell-1+\sum_{m=0}^{k+1} \theta^{m}\Big)
        = 1- \lambda\eta,
\end{align*}
against the hypothesis that $\lambda>\eta^{-1}$.

The above contradiction shows that $j=k$.  Thus $\cP(k,\ell)$ holds and so
$\cP(k+1,\ell-1)$ holds as well (by Proposition \ref{minimal:prop:Pjl}).
As $(k+1)+2(\ell-1)=n-1<n$, Corollary \ref{first:cor:T} provides arbitrarily large
values of $i$ for which
\[
  1  \ll Y_{k}(i)^{1-(\ell-1)\lambda}
        \Big(\prod_{m=0}^{k} Y_{k-m}(i)^{-\lambda}\Big) Y_{-1}(i)^{-\lambda}.
\]
Since $j=k\ge 1$, \eqref{main:prop:eq2} applies.  So, we conclude that
\[
  0 \le  1 - \lambda\Big( \ell-1+\sum_{m=0}^{k+1} \theta^{m} \Big)
      = 1 - \lambda\eta,
\]
which again contradicts the hypothesis that $\lambda>\eta^{-1}$.
\end{proof}

\begin{proof}[Proof of Theorem \ref{intro:thm:main}]
For $2\le n\le 3$, the upper bound for $\hlambda_n(\xi)$ provided by the
theorem is weaker than the prior ones mentioned in the introduction.  For
$4\le n\le 11$, they are also weaker than those coming from Theorems
\ref{intro:thm:pair} and \ref{intro:thm:impair}, and listed on Table
\ref{intro:table1}. So, we may assume that $n\ge 12$.

Suppose, by contradiction, that the upper bound for $\hlambda_n(\xi)$ given
by Theorem \ref{intro:thm:main} does not hold for such $n$.  Then, we may
assume that
\[
 \lambda=\Big(\frac{n}{2}+a\sqrt{n}+\frac{1}{3}\Big)^{-1}
\]
where $a=(1-\log(2))/2$.  Define
\begin{align*}
  &\ell =\left\lfloor \frac{n}{2}-\frac{\log(2)}{2}\sqrt{n} + 1\right\rfloor,
  \quad
  k= n -2\ell,
  \quad
  \theta= \frac{\ell\lambda}{1-\lambda}
  \et
  \eta= \ell - 1 + \sum_{m=0}^{k+1}\theta^m.
\label{main:proof:eq:eta}
\end{align*}
Since $n\ge 9$, we find that $1\le \ell <n/2$ and so $k\ge 1$.
We will show further that
\begin{equation}
 \label{main:proof:eq:conditions}
 \frac{1}{2}\le \theta^k <1
  \et
 \eta>\lambda^{-1}.
\end{equation}
So this choice of parameters fulfills all the hypotheses of
Proposition \ref{main:prop} but not its conclusion.  This will
conclude the proof, by contradiction.

A quick computer computation shows that
\eqref{main:proof:eq:conditions} holds for $12\le n< 900$.  So, we
may assume that $\sqrt{n}\ge 30$.  This assumption will simplify the
estimates below.

By choice of $\ell$, there is some $t\in(0,2]$ such that
\[
  \ell=(n-\log(2)\sqrt{n} + t)/2\,,
  \quad\hbox{and then}\quad
   k=\log(2) \sqrt{n} - t.
\]
Using the actual value of $\lambda$, this gives
\[
  \theta=\frac{n-k}{2\big(\lambda^{-1}-1\big)}
     =\frac{n-\log(2)\sqrt{n}+t}{n+2a\sqrt{n}-4/3}.
\]
Define
\begin{equation}
 \label{main:proof:eq:epsilon}
 \epsilon=1-\theta=\frac{\sqrt{n}-(4/3+t)}{n+2a\sqrt{n}-4/3}.
\end{equation}
As $n\ge 12$, we have $0<\epsilon<1$, thus $0<\theta<1$, and so $\theta^k<1$.
We set
\[
  r=-\frac{\log(2)}{\log(1-\epsilon)}
\]
so that $\theta^r=(1-\epsilon)^r=1/2$.  Since $0<\epsilon<1$, we
find
\[
  \frac{\epsilon}{1-\epsilon/2}
   =\sum_{i=1}^\infty\frac{\epsilon^i}{2^{i-1}}
   \le \sum_{i=1}^\infty\frac{\epsilon^i}{i}
   =-\log(1-\epsilon)
   \le \sum_{i=1}^\infty \epsilon^i
   =\frac{\epsilon}{1-\epsilon},
 \]
thus, by definition of $r$,
\begin{equation}
 \label{main:proof:estimates:r}
  \log(2)\left(\frac{1}{\epsilon}-1\right)
  \le r
  \le \log(2)\left(\frac{1}{\epsilon}-\frac{1}{2}\right).
\end{equation}
Moreover, \eqref{main:proof:eq:epsilon} yields
\[
  \frac{1}{\epsilon}
    = \sqrt{n} + b + \frac{c}{\sqrt{n}-(4/3+t)}
\]
with $b=2a+4/3+t$ and $c=b(4/3+t)-4/3$.  As $c>0$, this implies that
\begin{equation}
 \label{main:proof:estimates:1/epsilon}
 \frac{1}{\epsilon}
    \ge \sqrt{n} + b + \frac{c}{\sqrt{n}} \ge \sqrt{n}+1.64,
\end{equation}
so $r>\log(2)\sqrt{n}\ge k$, and thus $\theta^k\ge \theta^r=1/2$.
This proves the first condition in \eqref{main:proof:eq:conditions}.

To verify the second condition $\eta>\lambda^{-1}$, we first note that
\begin{equation}
 \label{main:proof:majoration:1/epsilon}
 \frac{1}{\epsilon}
    \le \sqrt{n} + b + \frac{c}{30-(4/3+2)} \le \sqrt{n}+4.05,
\end{equation}
using the hypothesis $\sqrt{n}\ge 30$ and $t\le 2$.   We set further
\[
  \delta=k+2-r \et E=\frac{1-\theta^\delta}{1-\theta}-\delta.
\]
As $\theta^r=1/2$, we find
\begin{align*}
  \eta
   =\ell-1+\frac{1-\theta^{k+2}}{1-\theta}
   &=\ell-1+\frac{1-\theta^{r}}{1-\theta}+\theta^r\frac{1-\theta^{\delta}}{1-\theta}\\
   &=\frac{n-k}{2}-1+\frac{1}{2\epsilon}+\frac{1}{2}(E+\delta)
   =\frac{n}{2}-\frac{r}{2}+\frac{1}{2\epsilon}+\frac{E}{2}.
\end{align*}
Using the upper bound for $r$ given by \eqref{main:proof:estimates:r}
and then the lower bound for $1/\epsilon$ given by
\eqref{main:proof:estimates:1/epsilon}, we deduce that
\begin{align*}
  \eta-\lambda^{-1}
   &\ge \frac{n}{2}+\frac{a}{\epsilon}+\frac{\log(2)}{4}+\frac{E}{2}
        -\left(\frac{n}{2}+a\sqrt{n}+\frac{1}{3}\right)\\
   &\ge 1.64\,a+\frac{\log(2)}{4}-\frac{1}{3}+\frac{E}{2}\\
   &\ge 0.09+\frac{E}{2}.
\end{align*}
So it remains to show that $E>-0.18$.

Lagrange remainder theorem gives
\[
  \theta^\delta
   = (1-\epsilon)^\delta
   = 1 - \delta\epsilon
      + \frac{\delta(\delta-1)}{2}(1-\epsilon')^{\delta-2}\epsilon^2
\]
for some real number $\epsilon'$ with $0<\epsilon'<\epsilon$, and thus
\[
  E = \frac{1-\theta^\delta}{\epsilon} - \delta
    = -\frac{\delta(\delta-1)}{2}(1-\epsilon')^{\delta-2}\epsilon.
\]
As $k<r$, we have $\delta<2$.  Moreover, the upper bound for $r$ given
by \eqref{main:proof:estimates:r} together with that of $1/\epsilon$
given by \eqref{main:proof:majoration:1/epsilon} yields
\[
  \delta \ge k+2-\log(2)\left(\frac{1}{\epsilon}-\frac{1}{2}\right)
         \ge -2.47.
\]
Since by \eqref{main:proof:estimates:1/epsilon} we have
$\epsilon <1/\sqrt{n}\le 1/30$, we conclude that
\[
  |E| \le \frac{2.47\times3.47}{2}\left(1-\frac{1}{30}\right)^{\disp -4.47}\frac{1}{30}
      \le 0.17.
\qedhere
\]
\end{proof}

%
%

\section{A new construction}
\label{sec:new}

In this section, we introduce a new construction which in some cases
yields $\lambda_k(\xi)>1$ for an integer $k\ge 1$.  We will use it
in the proof of Theorems \ref{intro:thm:impair} and \ref{intro:thm:pair}
in combination with the following results of Schleischitz
\cite[Theorems 1.6 and 1.12]{Schl2016}.

\begin{theorem}[Schleischitz, 2016]
\label{new:thm:Schleischitz}
Let $\xi\in\bR\setminus\bQ$.  For each integer $n\ge 1$, we have
\[
 \hlambda_n(\xi) \le \max\{1/n,1/\lambda_1(\xi)\}.
\]
Moreover, if $\lambda_k(\xi)>1$ for some integer $k\ge 1$, then
$\lambda_1(\xi)=k-1+k\lambda_k(\xi)$.
\end{theorem}

In fact, we will simply need the following weaker consequence.

\begin{cor}
\label{new:cor}
Let $\xi\in\bR\setminus\bQ$ and let $\lambda\in(1/n,1)$.  Suppose that, for some
integer $k\ge 1$, there exist non-zero points $C\in\bZ^{k+1}$ for which
$L_\xi(C)$ is arbitrarily small while the products $\norm{C}L_\xi(C)^\lambda$
remain bounded from above.  Then,  we have $\hlambda_n(\xi)\le \lambda$.
\end{cor}

\begin{proof}
The hypothesis implies that $\lambda_k(\xi)\ge 1/\lambda>1$.  By the above
theorem, we conclude that  $\lambda_1(\xi)\ge 1/\lambda$ and so
$\hlambda_n(\xi)\le\max\{1/n,\lambda\}=\lambda$.
\end{proof}

For each positive integer $k$ and each non-zero subspace $V$ of $\bR^{k+1}$
defined over $\bQ$, it is natural to define
\[
 L_\xi(V)=\norm{\uz_1\wedge\cdots\wedge\uz_s\wedge\Xi_k}
\]
where $\{\uz_1,\dots,\uz_s\}$ is a basis of $V\cap\bZ^{k+1}$ over
$\bZ$, and where $\Xi_k=(1,\xi,\dots,\xi^k)$.  This is independent of the choice
of the basis, like for the height $H(V)=\norm{\uz_1\wedge\cdots\wedge\uz_s}$
of $V$.  In particular, if $\uz$ is a primitive point of $\bZ^{k+1}$, we find
\[
 L_\xi(\Span{\uz})=\norm{\uz\wedge\Xi_k}\asymp L_\xi(\uz)
\]
with an implied constant that depends only on $k$ and $\xi$ (see \S\ref{sec:heights}).
In general, if $\{\uy_1,\dots,\uy_s\}$ is a maximal linearly independent
subset of $V\cap\bZ^{k+1}$, then arguing as in the proof
of Lemma \ref{heights:lemma}, we find
\[
 L_\xi(V)\le\norm{\uy_1\wedge\cdots\wedge\uy_s\wedge\Xi_k}
   \ll L_\xi(\uy_1)\cdots L_\xi(\uy_s),
\]
with an implied constant of the same nature.  We can now present our construction.

\begin{prop}
\label{new:prop}
Let $\xi\in\bR\setminus\bQ$, let $k,\ell\ge 1$ be integers, and let $n=k+\ell$.
Suppose that $V$ is a subspace of $\bR^{k+1}$ of dimension $k$, and that
$\ux\in\bZ^{n+1}$ satisfies $\UU^\ell(\ux)\not\subseteq V$.  Finally, let
$\{\uz_1,\dots,\uz_k\}$ be a basis of $V\cap \bZ^{k+1}$.  Then the point
\[
 C=\left(\det(\uz_1,\dots,\uz_k,\ux^{(0,\ell)}),
             \dots,
             \det(\uz_1,\dots,\uz_k,\ux^{(\ell,\ell)})\right)\in\bZ^{\ell+1}
\]
is non-zero.  It satisfies
\[
 \norm{C} \ll \norm{\ux}L_\xi(V)+H(V)L_\xi(\ux)
 \et
  L_\xi(C) \ll H(V)L_\xi(\ux)
\]
with implied constants that depend only on $k$ and $\xi$.
\end{prop}

We will write $C(V,\ux)$ to denote this point $C$, although it is determined
by $V$ and $\ux$ only up to multiplication by $\pm 1$.  In practice, this is
no problem since this ambiguity does not affect the quantities $\norm{C}$
and $L_\xi(C)$.

\begin{proof}  Write $\ux=(x_0,\dots,x_n)$ and $C=(C_0,\dots,C_\ell)$.
Then we have $\ux=x_0\Xi_n+\Delta$ with $\norm{\Delta}\asymp L_\xi(\ux)$
and for each $j=0,\dots,\ell$ we find
\[
 C_j = x_0\xi^j\det(\uz_1,\dots,\uz_k,\Xi_k) + \det(\uz_1,\dots,\uz_k,\Delta^{(j,\ell)}).
\]
The estimates for $\norm{C}$ and $L_\xi(C)=\max_{1\le j\le k}|C_j-C_0\xi^j|$ follow.
\end{proof}

%
%

\section{Small odd degree}
\label{sec:odd}

This section is devoted to the proof of Theorem~\ref{intro:thm:impair}.
So, we suppose that
\[
 n=2m+1 \geq 5
\]
is an odd integer.  We argue by contradiction, assuming that $\hlambda_n(\xi)
> \alpha$ where $\alpha=\alpha_m$ is the unique positive root of
\[
  P_m(x)=1-(m+1)x-mx^2.
\]
Then $\cP(1,m)$ holds by Corollary \ref{P1:cor}, and so $\cP(2,m-1)$
holds as well by Proposition \ref{minimal:prop:Pjl}.  Thus, for each large
enough $i\in I$, the vector space
\[
 V_i=\UU^{m-1}(\ux_{i-1},\ux_i,\ux_{i+1}) \subseteq \bR^{m+3}
\]
has dimension at least $m+2$ (see the definitions and remarks in Section \ref{sec:minimal}).

\begin{lemma}
\label{odd:lemma1}
With the above hypotheses, we have
\begin{equation}
\label{odd:lemma1:eq1}
 X_{j+1}^\theta\ll X_{i+1}
 \et
 X_{i+1}^\theta\ll X_i
 \quad
 \text{where}
 \quad
 \theta=\frac{m\alpha}{1-\alpha}=\frac{1}{1+\alpha},
\end{equation}
for each pair of consecutive elements $i<j$ of $I$.  If $i$ is large enough,
then $\dim(V_i)=m+2$,
\begin{equation}
\label{odd:lemma1:eq2}
 H(V_i)\ll X_{j+1}^{-(m-1)\alpha}X_{i+1}^{1-\alpha}X_i^{-\alpha}
 \et
 L_\xi(V_i)\ll X_{j+1}^{-m\alpha}X_{i+1}^{-\alpha}X_i^{-\alpha}.
\end{equation}
\end{lemma}

\begin{proof}
Since $\hlambda_n(\xi) > \alpha$, we can choose $\epsilon>0$ such that
$\alpha+\epsilon<\hlambda_n(\xi)$.  Then the results of the preceding sections
apply with $\lambda=\alpha+\epsilon$.   In particular, for each
pair of consecutive elements $i<j$ of $I$, Lemma \ref{P1:lemma2} gives
\[
 H(\UU^m(\ux_i,\ux_{i+1}))
   \ll X_{j+1}^{-m\alpha}X_{i+1}^{1-\alpha}
 \et
 H(V_i) \ll X_{j+1}^{-(m-1)\alpha}X_{i+1}^{1-\alpha-\epsilon}X_i^{-e(i)\alpha}
\]
where $e(i)=\dim(V_i)-(m+1)$, because both $\cP(1,m)$ and $\cP(1,m-1)$ hold.
Since we have $H(\UU^m(\ux_i,\ux_{i+1}))\ge 1$, the first estimate yields
$X_{j+1}^\theta\ll X_{i+1}$ with $\theta$ as in \eqref{odd:lemma1:eq1}
(this also follows from Corollary \ref{alt:cor:F}).  So, if $i$ is large enough to
admit a predecessor $h<i$ in $I$, we also have $X_{i+1}^\theta\ll X_{h+1}\le X_i$,
thus $X_{i+1}\ll X_i^{1+\alpha}$.  This proves \eqref{odd:lemma1:eq1}.
If furthermore $V_i=\bR^{m+3}$, then $e(i)=2$ and using $X_{j+1}\ge X_{i+1}$,
we obtain
\[
 1=H(V_i) \ll X_{i+1}^{1-m\alpha-\epsilon}X_i^{-2\alpha}
   \ll X_i^{(1+\alpha)(1-m\alpha-\epsilon)-2\alpha}
    =X_i^{-(1+\alpha)\epsilon},
\]
which forces $i$ to be bounded.  So, if $i$ is large enough, we have $\dim(V_i)=m+2$,
thus $e(i)=1$ and the estimate for $H(V_i)$ in \eqref{odd:lemma1:eq2} follows.
Finally, the proof of Lemma \ref{P1:lemma2} shows that $V_i$ admits a basis
of the form
\[
 \big\{\ux_j^{(0,m-1)},\dots,\ux_j^{(m-1,m-1)},\ux_{j-1}^{(p,m-1)},
      \ux_h^{(q,m-1)}\big\}
\]
for some $p,q\in\{0,1,\dots,m-1\}$ and some $h\in\{i-1,j-1\}$. So
the considerations of the preceding section provide
$L_\xi(V_i)   \ll L_j^mL_{j-1}L_{i-1}
\le X_{j+1}^{-m\alpha}X_{i+1}^{-\alpha}X_i^{-\alpha}$.
\end{proof}

It is now an easy matter to complete the proof of Theorem \ref{intro:thm:impair}.
By the preceding lemma, there are arbitrarily large pairs of successive elements
$i<j$ of $I$ for which $\dim(V_i)=m+2$ and $V_j\not\subseteq V_i$.  The latter
condition means that $\UU^{m-1}(\ux_{j+1})\not\subseteq V_i$.  So, for these
pairs, Proposition \ref{new:prop} shows that the point
\[
 C_i=C(V_i,\ux_{j+1})\in\bZ^m
\]
is non-zero.  Using the estimates of the preceding lemma, it also gives
\begin{align*}
 L_\xi(C_i)
   &\ll H(V_i)L_{j+1}
   \ll X_{j+1}^{-m\alpha}X_{i+1}^{1-\alpha}X_i^{-\alpha},\\
 \norm{C_i}
   &\ll X_{j+1}L_\xi(V_i)+H(V_i)L_{j+1}
    \ll X_{j+1}^{1-m\alpha}X_{i+1}^{-\alpha}X_i^{-\alpha}.
\end{align*}
Using \eqref{odd:lemma1:eq1}, we find that
\begin{align*}
  \norm{C_i}L_\xi(C_i)^\alpha
    \ll X_{j+1}^{1-m\alpha-m\alpha^2}X_{i+1}^{-\alpha^2}X_i^{-\alpha-\alpha^2}
     &= X_{j+1}^{\alpha}X_{i+1}^{-\alpha^2}X_i^{-\alpha/\theta}\\
    &\ll X_{j+1}^{\alpha}X_{i+1}^{-\alpha^2-\alpha}
     = X_{j+1}^{\alpha}X_{i+1}^{-\alpha/\theta}
    \ll 1
\end{align*}
is bounded from above, and that
\begin{align*}
  L_\xi(C_i)
    \ll X_{j+1}^{-m\alpha}X_{i+1}^{1-\alpha}X_i^{-\alpha}
    \ll X_{j+1}^{-m\alpha}X_{i+1}^{1-\alpha-\alpha\theta}
    \le X_{j+1}^{-m\alpha+1-\alpha-\alpha\theta}
     = X_{j+1}^{-\alpha^2\theta}
\end{align*}
tends to $0$ as $i$ goes to infinity.  By Corollary \ref{new:cor}, this implies that
$\hlambda_n(\xi)\le \alpha$.

%
%

\section{Small even degree}
\label{sec:even}

We conclude, in this section, with the proof of Theorem~\ref{intro:thm:pair}.
So, we assume that
\[
 n=2m \geq 4
\]
is an even integer.  For the proof, define $\beta=\beta_m$ to be the single positive root of
\begin{align*}
 Q_m(x)&=\begin{cases}
    1-mx-mx^2-m(m-1)x^3 &\text{if $m\ge3$,}\\
    1-3x+x^2-2x^3-2x^4 &\text{if $m=2$,}
  \end{cases}\\
\intertext{as in the statement of the theorem, then write}
 \gamma&=\begin{cases}
    \disp \frac{m+4}{m^2+6m+2} &\text{if $m\ge4$,}\\[8pt]
    8/33 &\text{if $m=3$,}\\
    1/3 &\text{if $m=2$,}
  \end{cases}\\
\intertext{and define $\delta$ to be the single positive root of}
 R_m(x)&=1-(m+1)x-(m-1)x^2.
\end{align*}
It is a simple matter to check that
\[
 1/(m+2) < \delta < \gamma < \beta.
\]
We will prove that $\hlambda_n(\xi)\le \beta$ through the following chains of implications
\[
\begin{array}{lclcl}
  \hlambda_n(\xi)>\delta
  &\Longrightarrow & \text{$\cP(1,m-1)$ holds}
  &\Longrightarrow & \text{$\cP(2,m-2)$ holds}\\[5pt]
  \hlambda_n(\xi)>\gamma
  &\Longrightarrow & \text{$\cP(2,m-1)$ does not hold}
  &\Longrightarrow & \hlambda_n(\xi)\le \beta.
\end{array}
\]
Recall that $\lambda$ represents a fixed real number with
$0<\lambda<\hlambda_n(\xi)$.

\begin{lemma}
\label{even:lemma1}
Suppose that $\hlambda_n(\xi) > \delta$.  Then both $\cP(1,m-1)$
and $\cP(2,m-2)$ hold.  Moreover, we have
$X_{i+1}^{m\lambda}\ll X_i$  for each $i\ge 0$.
\end{lemma}

\begin{proof}
For the choice of $\ell=m-1$, we have $1\le \ell<n/2$ and
$R_m(x)=1-(n-\ell)x-\ell x^2$.  Thus, by Corollary \ref{P1:cor}, our
hypothesis implies $\cP(1,m-1)$ which in turn implies $\cP(2,m-2)$,
by the general Proposition \ref{minimal:prop:Pjl}.  Since
$\delta\ge 1/(m+2)\ge 1/(2m)$, the growth estimate follows
from Lemma \ref{P1:lemma3}.
\end{proof}

\begin{lemma}
\label{even:lemma2}
Suppose that $\cP(2,m-1)$ holds and choose $\theta\in(0,1]$ such that
\begin{equation}
\label{even:lemma2:eq}
 \frac{m-1}{\theta}+\theta > \frac{1}{\lambda}-1.
\end{equation}
Then we have $X_{j+1}^\theta \le X_{i+1}$ for any large enough pair of consecutive
elements $i<j$ of $I$.
\end{lemma}

\begin{proof}
Property $\cP(2,m-1)$ means that $\UU^{m-1}(\ux_{i-1},\ux_i,\ux_{i+1})=\bR^{m+2}$
for each large enough $i\in I$.   By definition, it also implies property $\cP(1,m-1)$.
Thus, for all but finitely many triples $h<i<j$ of consecutive elements of $I$, we
obtain, by Lemma \ref{P1:lemma2},
\[
 1=H(\UU^{m-1}(\ux_{i-1},\ux_i,\ux_{i+1}))
   \ll X_{j+1}^{-(m-1)\lambda}X_{i+1}^{1-\lambda}X_{h+1}^{-\lambda}
\]
using the crude estimate $X_i\ge X_{h+1}$.  Taking logarithms, this gives
\[
 (m-1)\frac{\log(X_{j+1})}{\log(X_{i+1})}+\frac{\log(X_{h+1})}{\log(X_{i+1})}
 \le \frac{1-\lambda}{\lambda} + \cO\Big(\frac{1}{\log(X_{i+1})}\Big).
\]
Let $\mu\in[0,1]$ denote the inferior limit of the ratio $\log(X_{i+1})/\log(X_{j+1})$
as $(i,j)$ runs through the pairs of consecutive elements $i<j$ of $I$.  Then, there
is a sequence of triples $(h,i,j)$, with $i$ going to infinity and $h<i<j$ consecutive
in $I$, for which the ratio  $\log(X_{i+1})/\log(X_{j+1})$ converges to $\mu$.
Over that sequence, the inferior limit of $\log(X_{h+1})/\log(X_{i+1})$ is at
least $\mu$, and so the above inequality implies that $\mu>0$ and
\[
 \frac{m-1}{\mu}+\mu \le \frac{1}{\lambda}-1.
\]
As $m\ge 2$, the expression $(m-1)/x+x$ is a strictly decreasing function
of $x$ on $(0,1]$.  We conclude that our choice of $\theta$ satisfies
$0<\theta<\mu$, and so $\theta \le \log(X_{i+1})/\log(X_{j+1})$ for
any pair of consecutive elements $i<j$ of $I$ with $i$ large enough.
\end{proof}

\begin{lemma}
\label{even:lemma3}
Suppose that $\hlambda_n(\xi)>\gamma$.  Then $\cP(2,m-1)$ does not hold.
\end{lemma}

\begin{proof}
Suppose on the contrary that $\cP(2,m-1)$ holds.  By the hypothesis,
we may assume that $\lambda>\gamma$.  A short
computation shows that $(m-1)/\theta+\theta\ge 1/\gamma-1$ for
\[
 \theta=\begin{cases}
   1 &\text{if $m=2$,}\\
   15/17 &\text{if $m=3$,}\\
   (m+4)/(m+5) &\text{if $m\ge 4$.}
   \end{cases}
\]
Thus this choice of $\theta$ fulfills the main condition \eqref{even:lemma2:eq}
of Lemma \ref{even:lemma2}, and so we have $X_{j+1}^\theta\le X_{i+1}$
for each large enough pair of consecutive elements $i<j$ of $I$.  For $m=2$, this
becomes $X_{j+1}\le X_{i+1}$ which
is already a contradiction because the sequence $(X_i)_{i\ge 0}$ is strictly increasing.
Thus, we may assume that $m\ge 3$.

Since $\cP(2,m-1)$ holds, Proposition \ref{minimal:prop:Pjl} implies
that  $\cP(3,m-2)$ holds as well.  So, for each large enough $i$,  the subspace
$V_i=\UU^{m-2}(A_3(i))$ of $\bR^{m+3}$ has dimension at least $m+2$
and Proposition \ref{first:prop} gives
\begin{equation}
\label{even:lemma3:eq1}
 H(V_i)\ll Y_2(i)^{1-(m-1)\lambda}Y_1(i)^{-\lambda}Y_0(i)^{-(1+e(i))\lambda}
\end{equation}
where $e(i)=\dim(V_i)-m-2\in\{0,1\}$.  If $V_i\neq \bR^{m+3}$ for infinitely many
$i$, then Lemma \ref{P0:lemma:A} provides arbitrarily large $i\in I$ for which
$\UU^{m-2}(\ux_{i-1})\not\subseteq V_i$.  For those $i$, we have $e(i)=0$ and
Lemma \ref{P0:lemma:S} gives $1\ll H(V_i)L_{i-1}$.  Then, we obtain
\begin{equation}
\label{even:lemma3:eq2}
 1 \ll Y_2(i)^{1-(m-1)\lambda}Y_1(i)^{-\lambda}Y_0(i)^{-\lambda}X_i^{-\lambda}.
\end{equation}
Otherwise, we have $e(i)=1$ for all sufficiently large $i\in I$ and the above estimate
follows directly from \eqref{even:lemma3:eq1} since $1\le H(V_i)$ and
$Y_0(i)=X_{i+1}>X_i$.  Thus \eqref{even:lemma3:eq2} holds for infinitely many
$i\in I$.  Viewing such $i$ as part of a triple of consecutive elements
$h<i<j$ of $I$, we have $Y_1(i)=X_{j+1}$, $Y_0(i)=X_{i+1}$ and
$X_i\ge X_{h+1}$, thus
\begin{equation}
\label{even:lemma3:eq3}
 Y_1(i)^\theta\ll Y_0(i) \et Y_0(i)^\theta\ll X_i
\end{equation}
by our initial observation at the beginning of the proof.
So, we deduce that
\begin{equation}
\label{even:lemma3:eq4}
 Y_2(i)^{1-(m-1)\lambda}
  \gg Y_1(i)^{\lambda(1+\theta+\theta^2)} \gg Y_1(i)^{3\lambda\theta}
\end{equation}
for arbitrarily large $i\in I$.  Since $m\ge 3$, we may also apply Corollary
\ref{alt:cor:E} with $j=2$ and $\ell=m-1$.  Using \eqref{even:lemma3:eq3},
this gives
\[
 Y_2(i)^{(m-2)\lambda}
  \ll Y_1(i)^{1-2\lambda}Y_0(i)^{-\lambda}
  \ll Y_1(i)^{1-2\lambda-\lambda\theta}
\]
for each $i\in I$.  Substituting this upper bound for $Y_2(i)$ into
\eqref{even:lemma3:eq4} and then comparing powers of $Y_1(i)$, we conclude
that
\[
 3(m-2)\lambda^2\theta \le (1-(m-1)\lambda)(1-2\lambda-\lambda\theta),
\]
and thus $3(m-2)\theta \le (1/\gamma-(m-1))(1/\gamma-2-\theta)$,
as $\lambda$ can be taken arbitrarily close to $\gamma$.  However, this inequality
is false for the actual values of $\gamma$ and $\theta$.  This contradiction shows
that $\cP(2,m-1)$ does not hold.
\end{proof}

\subsection*{Proof of Theorem \ref{intro:thm:pair}}
We may assume that $\gamma<\lambda<\hlambda_n(\xi)$.  Then $\cP(1,m-1)$ and
$\cP(2,m-2)$ hold while $\cP(2,m-1)$ does not hold, by Lemmas \ref{even:lemma1}
and \ref{even:lemma3}.   In particular, Proposition \ref{P0:prop:I} applies
with $j=2$ and $\ell=m-1$, and so there are infinitely many integers $i\ge 1$
for which
\begin{equation}
\label{even:proof:eq1}
  \dim\UU^{m-1}(A_2(i))=m+1
  \et
  \UU^{m-1}(\ux_{i-1})
     \not\subseteq \UU^{m-1}(A_2(i))
     \varsubsetneq \bR^{m+2}.
\end{equation}
For those $i$, Lemma \ref{P0:lemma:S} and Proposition \ref{P0:prop:I}
further give
\[
 1\ll H(\UU^{m-1}(A_2(i)))L_{i-1}
  \et
  1\ll H(\UU^{m-2}(A_2(i)))L_{i-1}^2.
\]
Any such $i$ belongs to $I$ and, upon denoting by $j$ its successor in $I$,
Proposition \ref{first:prop} gives
\[
  H(\UU^{m-2}(A_2(i)))
    \ll Y_1(i)^{1-(m-1)\lambda}Y_0(i)^{-\lambda}
     = X_{j+1}^{1-(m-1)\lambda}X_{i+1}^{-\lambda}.
\]
By $\cP(1,m-1)$, we also have $\dim\UU^{m-1}(A_1(i))\ge m+1$
if $i$ is large enough.  Comparing with \eqref{even:proof:eq1}, this
implies that $\UU^{m-1}(A_2(i))=\UU^{m-1}(A_1(i))$ and so Lemma
\ref{P1:lemma2} gives
\[
  H(\UU^{m-1}(A_2(i)))
   =H(\UU^{m-1}(A_1(i)))
    \ll X_{j+1}^{-(m-1)\lambda}X_{i+1}^{1-\lambda}.
\]
By Lemma \ref{even:lemma1}, we also have $L_{i-1}\ll X_i^{-\lambda}
\ll X_{i+1}^{-m\lambda^2}$.  Combining all the above inequalities, we obtain
\begin{equation}
\label{even:proof:eq2}
  1\ll X_{j+1}^{-(m-1)\lambda}X_{i+1}^{1-\lambda-m\lambda^2}
  \et
  1\ll X_{j+1}^{1-(m-1)\lambda}X_{i+1}^{-\lambda-2m\lambda^2}.
\end{equation}
As we can take $i$ arbitrarily large, this in turn implies that
\[
 0\le (1-(m-1)\lambda)(1-\lambda-m\lambda^2)-(m-1)\lambda(\lambda+2m\lambda^2).
\]
If $m\ge 3$, the right hand side of this inequality simplifies to $Q_m(\lambda)$.
So, in that case, we obtain $\lambda\le \beta$, thus $\hlambda_n(\xi)\le \beta=\beta_m$
as needed.

For the case $m=2$, we look more closely at the vector spaces
\[
 V_i=\UU^1(A_1(i-1))=\UU^1(\ux_{i-1},\ux_{i})\subseteq\bR^4
\]
for each integer $i\ge 1$.  Since $\cP(1,1)$ holds, we have $\dim\UU^1(\ux_i)=2$
and $\dim(V_i)\ge 3$ for each large enough $i$.  When $V_i=\bR^4$, we find
$1=H(V_i)\ll X_iL_{i-1}^3\ll X_i^{1-3\lambda}$ since $V_i$ is generated by
points $\uy\in\bZ^4$ with $\norm{\uy}\le X_i$ and $L_\xi(\uy)\ll L_{i-1}$.
As $\lambda>\gamma=1/3$, we conclude that both $\dim\UU^1(\ux_i)=2$
and $\dim(V_i)=3$ for each large enough $i$, say for $i\ge i_0$.  Then, $V_i$
admits a basis of the form $\{\ux_{i-1}^{(p,1)},\ux_i^{(0,1)},\ux_i^{(1,1)}\}$
for some $p\in\{0,1\}$, thus
\begin{equation*}
\label{even:proof:eq3}
 H(V_i)\ll X_iL_iL_{i-1}\ll X_{i+1}^{-\lambda}X_i^{1-\lambda}
 \et
 L_\xi(V_i)\ll L_i^2L_{i-1}\ll X_{i+1}^{-2\lambda}X_i^{-\lambda}.
\end{equation*}
For each $i\ge i_0$ for which \eqref{even:proof:eq1} holds, we have
$\UU^1(\ux_{i-1})\not\subseteq V_{i+1}$, thus $V_i\cap V_{i+1}=\UU^1(\ux_i)$
and so $\UU^1(\ux_{i+1})\not\subseteq V_i$.  Then, by Proposition \ref{new:prop},
the point $C_i=C(V_i,\ux_{i+1})\in\bZ^2$ is non-zero with
\[
 \norm{C_i}
   \ll H(V_i)L_{i+1}+X_{i+1}L_\xi(V_i)
   \ll X_{i+1}^{1-2\lambda}X_i^{-\lambda}.
\]
Moreover, by Lemma \ref{minimal:lemma:HA1}, the pair $\{\ux_i,\ux_{i+1}\}$
is a basis of $A_1(i)\cap\bZ^5$ over $\bZ$. So, letting $j$ denote the successor
of $i$ in $I$, we may write $\ux_j=a\ux_i+b\ux_{i+1}$ for some $a,b\in\bZ$
with $b\neq 0$.  Since $C(V_i,\ux)$ is linear in $\ux$ (for a fixed basis of $V_i$)
and since $C(V_i,\ux_i)=0$, we find that $C(V_i,\ux_j)=b C_i$.  Thus, by
Proposition \ref{new:prop}, we obtain
\[
 L_\xi(C_i)
  \le L_\xi(C(V_i,\ux_j))
  \ll H(V_i)L_j
  \ll X_{j+1}^{-\lambda}X_{i+1}^{-\lambda}X_i^{1-\lambda}.
\]
In particular, $L_\xi(C_i)\ll X_i^{1-3\lambda}$ converges to $0$ as $i\to\infty$, since
$\lambda>1/3$.  Thus, by Corollary \ref{new:cor}, the product $\norm{C_i}L_\xi(C_i)^{\lambda}$ tends to infinity with $i$.  So we have
\[
 1 \ll \norm{C_i}L_\xi(C_i)^{\lambda}
   \ll X_{j+1}^{-\lambda^2}X_{i+1}^{1-2\lambda-\lambda^2}X_i^{-\lambda^2}.
\]
Using the estimate  $X_{i+1}^{2\lambda}\ll X_i$ from Lemma \ref{even:lemma1},
we conclude that
\begin{equation}
\label{even:proof:eq4}
 X_{j+1}^{\lambda^2} \ll X_{i+1}^{1-2\lambda-\lambda^2-2\lambda^3}
\end{equation}
for each pair of consecutive elements $i<j$ of $I$ with $i\ge i_0$, for which
\eqref{even:proof:eq1} holds.  For these, the two estimates \eqref{even:proof:eq2}
also apply.  In particular, the second one yields
\[
  X_{i+1}^{\lambda+4\lambda^2} \ll X_{j+1}^{1-\lambda}.
\]
Combining this with \eqref{even:proof:eq4}, we conclude that
$\lambda^2(\lambda+4\lambda^2)\le (1-\lambda)(1-2\lambda-\lambda^2-2\lambda^3)$,
which simplifies to $Q_2(\lambda)\ge 0$.  This gives $\lambda\le \beta$ and thus
$\hlambda_4(\xi)\le \beta=\beta_2$.


\end{document}